\documentclass{imsart}

\RequirePackage{amsthm,amsmath,mathrsfs, amssymb, bbm}
\RequirePackage[numbers]{natbib}
\RequirePackage[colorlinks,citecolor=blue,urlcolor=blue]{hyperref}
\RequirePackage{tikz}
\usetikzlibrary{patterns,decorations.pathreplacing}

\newcommand{\prob}{\mathbb{P}}

\newcommand{\exptn}{\mathbb{E}}

\newcommand{\dtv}{{\rm d}}

\newcommand{\unif}{\mathsf{U}}

\newcommand{\mbbo}{\mathbbm{1}}

\newcommand{\bdelta}{\boldsymbol{\delta}}

\DeclareMathOperator*{\argmin}{argmin}

\newcommand{\sss}{\scriptscriptstyle}
\newcommand{\norm}[1]{\| #1 \|}

\newcommand{\inv}{{-1}}
\newcommand{\finv}{F^{\leftarrow}}
\newcommand{\probconv}{\stackrel{\sss\mathbb{P}}{\longrightarrow}}
\newcommand{\asconv}{\stackrel{a.s.}{\longrightarrow}}
\newcommand{\vconv}{\stackrel{v}{\to}}
\newcommand{\eqd}{\stackrel{d}{=}}

\startlocaldefs
\usepackage{cleveref}
\newtheorem{thm}{Theorem}[section]
\newtheorem{propn}[thm]{Proposition}
\newtheorem{lemma}[thm]{Lemma}
\newtheorem{cor}[thm]{Corollary}

\theoremstyle{remark}
\newtheorem{remark}[thm]{Remark}

\newtheorem{example}{Example}[section]

\theoremstyle{definition}
\newtheorem{defn}[thm]{Definition}

\endlocaldefs

\allowdisplaybreaks
\numberwithin{equation}{section}

\newcommand{\PLFIT}{PLFit}
\newcommand{\e}{{\mathrm{e}}}
\newcommand{\TERM}[1]{{\sf{T}}^{\sss(#1)}}
\newcommand{\eqn}[1]{\begin{equation}#1\end{equation}}
\newcommand{\eqan}[1]{\begin{align}#1\end{align}}
\newcommand{\vep}{\varepsilon}
\newcommand{\nn}{\nonumber}

\newcommand{\Iterm}[1]{{\rm I}_n^{\sss(#1)}}

\usepackage{color}
\usepackage{todonotes}
\newcommand{\RvdH}[1]{\todo[color=magenta, inline]{Remco: #1}}
\newcommand{\BZ}[1]{\todo[color=green, inline]{Bert: #1}}

\begin{document}

\begin{frontmatter}
\title{Consistency of the \PLFIT{} estimator\\ for power-law data}
\runtitle{Consistency of \PLFIT{}}

\begin{aug}
\author{\fnms{Ayan} \snm{Bhattacharya} \thanksref{t1,t2}
\ead[label=e1]{ayanbhattacharya.isi@gmail.com}}
\author{\fnms{Bohan} \snm{Chen} \thanksref{t2}
\ead[label=e2]{chenbohan1988@gmail.com}}
\author{\fnms{Remco} \snm{van der Hofstad} \thanksref{t3}
\ead[label=e3]{r.w.v.d.hofstad@tue.nl}}
\and
\author{\fnms{Bert} \snm{Zwart}\thanksref{t2}
\ead[label=e4]{Bert.Zwart@cwi.nl}}

\thankstext{t1}{Partially supported by Polish National Science Centre Grant
\#
2018/29/B/ST1/00756 (2019-2022)
}
\thankstext{t2}{Partially supported by the Netherlands Organisation for Scientific Research (NWO) VICI grant \#
639.033.413}
\thankstext{t3}{Partially supported by the Netherlands Organisation for Scientific Research (NWO) through the Gravitation Networks grant 024.002.003.  }

\runauthor{Bhattacharya, Chen, v.d.\ Hofstad and Zwart}

\affiliation{Wroc\l{a}w University of Science and Technology\thanksmark{t1}, Centrum Wiskunde \& Informatica\thanksmark{t2} and Eindhoven University of Technology\thanksmark{t3}.}

\address{Wroc{\l}aw University of Science and Technology\\
Department of Applied Mathematics\\
Faculty of Pure and Applied Mathematics\\
wyb. Stanis{\l}awa Wyspia\'{n}skiego 27, 50-370 Wroc{\l}aw,
Poland\\
 \printead{e1} }

\address{Centrum Wiskunde \& Informatica\\
P.O. Box 94079\\
1090 GB Amsterdam, Netherlands\\
\printead{e2}\\
\printead{e4}}

\address{Eindhoven University of Technology\\
PO Box 513\\
5600 MB Eindhoven, Netherlands\\
\printead{e3}}

\end{aug}

\begin{abstract}
We prove the consistency of the Power-Law Fit (\PLFIT{}) method proposed by Clauset et al.\ \cite{clauset:shalizi:newman:2009} to estimate the power-law exponent in data coming from a distribution function with regularly-varying tail. In the complex systems community, \PLFIT{} has emerged as the method of choice to estimate the power-law exponent.
Yet, its mathematical properties are still poorly understood. 

The difficulty in \PLFIT{} is that it is a minimum-distance estimator. It first chooses a threshold that minimizes the Kolmogorov-Smirnov distance between the data points larger than the threshold and the Pareto tail, and then applies the Hill estimator to this restricted data. Since the number of order statistics used is {\em random}, the general theory of consistency of power-law exponents from extreme value theory does not apply. Our proof consists in first showing that the Hill estimator is consistent for general intermediate sequences for the number of order statistics used, even when that number is random. Here, we call a sequence {\em intermediate} when it grows to infinity, while remaining much smaller than the sample size. The second, and most involved, step is to prove that the optimizer in \PLFIT{} is with high probability an intermediate sequence, unless the distribution has a Pareto tail above a certain value. For the latter special case, we give a separate proof.
\end{abstract}

\begin{keyword}[class=MSC]
\kwd[Primary ]{62G20}
\kwd[; secondary ]{62F10, 62F12, 62F20, 62G30}
\end{keyword}

\begin{keyword}
Regular variation, Heavy-tailed distribution, Consistency, Minimum-distance estimator, \PLFIT{} estimator
\end{keyword}

\end{frontmatter}


\section{Introduction and motivation}
\label{sec-intro}

The Power-Law Fit (\PLFIT{}) method, proposed by Clauset et al.\ \cite{clauset:shalizi:newman:2009}, has emerged as a popular method 
to estimate the power-law exponent in data coming from a distribution function with regularly varying tail. Despite its popularity, the mathematical foundations of it have not received much attention. In particular, it is not yet known whether the estimate arising from this method is in general consistent. This is contrary to many other methods for estimating the power-law exponent, see e.g.\ \cite{voitalov:vanderhoorn:vanderhofstad:krioukov:2018} for a detailed introduction in the methods for estimating power-law exponents and their rigorous properties.


Specifically,
\label{sec-motiv}
let $F\colon [0, \infty) \to [0,1]$ be a distribution function of a non-negative random variable, satisfying that, for every $x > 0$,
	\begin{align}
	\lim_{t \to \infty} \frac{1 - F(t x)}{1 - F(t)} = x^{-\alpha}. \label{eq:defn:regvar:tail:dist}
	\end{align}
Then we say that $F$ has a regularly varying tail, or equivalently that $\overline{F} = 1 - F$ is a regularly varying tail distribution. We write $\overline{F} \in {\rm RV}_{-\alpha}$ to denote that $F$ satisfies \eqref{eq:defn:regvar:tail:dist}. It follows from Karamata's theorem (\cite[Theorem~2.1 and Corollary~2.1(ii)]{resnick:2007}) that
	\eqn{
	\label{L-def-RV}
	\overline{F}(x) = x^{-\alpha} L(x),
	}
where $L \in {\rm RV}_0$ is a slowly varying function.  The Pareto, log-gamma and Cauchy distributions are some of the well-known families of distributions with a regularly varying tail (bear in mind though that we restrict to non-negative random variables here). Throughout the article, we assume that $F$ is supported on $[1, \infty )$ and is continuous. 

We next explain the \PLFIT{} method. In a fundamental work \cite{hill:1975}, Bruce M. Hill proposed a consistent estimator for the index of regular variation which is popularly known as Hill's estimator. Let $(X_i)_{i=1}^n$ be a random sample drawn from a distribution function $F$ such that $\overline{F}  \in {\rm RV}_{-\alpha}$.  Let $X_{1:n} \le X_{2:n} \le \ldots \le X_{n: n}$ denote the order statistics. We define Hill's estimator as
	\begin{align}
	\widehat{\alpha}_{n, k} := H^\inv_{n, k} = \Big( \frac{1}{k}  \sum_{i=1}^{k} \log \frac{X_{n- k + i : n}}{X_{n-k: n}} \Big)^\inv,
	\end{align}
and note that $\widehat{\alpha}_{n, k}$ is based on the $k \in \{1,2, \ldots, n-1 \}$ largest observations. Hill's estimator does not use the full data but only the largest few observations as they have the most information on the tail of the distribution.  Several other consistent estimators are known in the literature, such as Pickand's estimator, kernel estimator, peaks-over-threshold (POT) estimator, Q-Q estimator, etc.\ (see \cite{resnick:2007} and \cite{beirlant:goegebeur:segers:teugels:2006} for nice surveys). Throughout this article, we restrict ourselves to the Hill estimator, as the \PLFIT{} method does so.

The main issue with Hill's estimator is that the choice of $k$, the number of largest order statistics to be used to estimate $\alpha$ is subjective. In the pure Pareto case, the choice $k=n$ gives us the Maximum Likelihood (MLE) estimator. However, in the semiparametric setting in \eqref{eq:defn:regvar:tail:dist}, the choice $k=n$ is obviously bad. Under the semiparametric assumption, the choice of $k_n$ plays an important role unless the population distribution is Pareto, which is unrealistic in many practical settings. In a seminal work \cite{mason:1982}, Mason has shown that $\widehat{\alpha}_{n, k}=H_{n, k_n}^{-1}$ estimates $\alpha$ consistently if $(k_n)_{n\geq 1}$ is an intermediate sequence.  By an intermediate sequence, we mean that $k_n \uparrow \infty$ and $k_n = o(n)$. However, there still are many choices for $k_n$, such as $k_n = n^{\beta}$ for some $\beta \in (0, 1)$, $k_n = (\log n)^\beta$, etc. One naive way to choose $k_n$ is by looking at the point of stabilization when the Hill estimator is plotted against the number of order statistics used. This is popularly known as the Hill plot. However, the fluctuations in the Hill plot often make this choice difficult  when $L$ is not a constant (see the so-called {\em Hill horror plot} \cite{resnick:2007}). There are many prescriptions to improve the plot (such as the smoothing Hill plot, alternative Hill plot etc.\ in \cite{resnick:2007}). The choice where $k_n$ minimizes the sum of squares of the asymptotic bias and its standard error 
has also been proposed in \cite{haeusler:teugels:1985,hall:1982,hall:welsh:1984}. Finally, a data-driven or data-adaptive choice of $k_n$ is proposed in \cite{hall:welsh:1985} under some parametric assumptions on $\overline{F}$. 
Several other choices of $k_n$ were proposed based on a double bootstrap method (see \cite{danileson:dehaan:peng:devries:2001, draisma:dehaan:peng:pereira:2000, voitalov:vanderhoorn:vanderhofstad:krioukov:2018, qi:2008} and the references therein).

Clauset et al.\ \cite{clauset:shalizi:newman:2009} propose to choose the cut-off  $k_n$ as the value that minimizes the Kolmorogorov-Smirnov distance between the fitted power-law and the empirical conditional distribution function associated to the $k_n$ largest observations, assuming that this distribution is pure Pareto. To explain this in more detail, define
	\begin{align}
	\label{eq:defn:DNK}
	D_{n, k} := \sup_{y \ge 1} \Big| \frac{1}{k} \sum_{i=1}^{k} \mbbo_{(y, \infty)} \Big( \frac{X_{n-k + i: n}}{X_{n- k : n}} \Big)- y^{-\widehat{\alpha}_{n, k}}  \Big|,
	\end{align}
which is the Kolmogorov distance between the empirical distribution of the sequence $(X_{n-k + i: n}/X_{n- k : n})_{i=1}^k$ and the pure Pareto distribution on $[1,\infty)$,
and let $\kappa_n^\star := \argmin_{1 \le k \le n} D_{n, k}$. Then, Clauset et al.\ \cite{clauset:shalizi:newman:2009} propose $\widehat{\alpha}_{n, \kappa_n^\star}$ to estimate $\alpha$. This estimation procedure will be referred to as the \PLFIT{} method, and we shall refer to $\widehat{\alpha}_{n, \kappa_n^\star}$ as the \PLFIT{} estimator of $\alpha$. The \PLFIT{} estimator gives a quantitative estimate for the choice of $k$ rather than a choice based on the eyeballing technique in the Hill plot.

The \PLFIT{} estimator has become popular across a broad range of academic disciplines, as exemplified by the large number of citations to \cite{clauset:shalizi:newman:2009}.
%
%
However, a mathematical justification for its use is still lacking in the literature. One of the main reasons for this is that (a) it is not clear that the choice of $k_n$ in the \PLFIT{} method is intermediate; and (b) the sequence is {\em random}, so that most proof techniques that show properties of the Hill estimator do not apply. Bear also in mind that the $y^{-\widehat{\alpha}_{n, k}}$ term present in \eqref{eq:defn:DNK} is related to the pure Pareto distribution, rather than the semi-parametric form in \eqref{eq:defn:regvar:tail:dist} that we rely on.

Consistency is the most basic desirable property of an estimator, and has not yet been properly addressed in the literature. It can be expected that \PLFIT{} is a consistent estimator when $F$ is a Pareto distribution function. Surprisingly, even this fact is far from being obvious and not yet known in the literature. Consistency would be an easy consequence of \cite[Corollary~2.2]{drees:janssen:resnick:wang:2018}, but the corollary uses some a.s.\ uniqueness property for the minimum of a certain Gaussian process, which still lacks a rigorous mathematical proof, even though simulations do confirm it (see \cite[Remark~2.3]{drees:janssen:resnick:wang:2018}).

A reason behind the lack of mathematical results is that \PLFIT{} is an example of a {\em minimum distance estimation} (MDE) procedure, in which a criterion function is minimized over the parameter space. The criterion function is a functional measuring the distance between the empirical distribution and a parametric family of distributions. Popular choices of criterion functions include the Cram\'er-von Mises, the Kolmogorov-Smirnov (KS), and the Anderson-Darling criterion. MDE is known in the statistics literature for its robustness, which means that the estimate is not much affected by the small departure of the population distribution from the parametric family.
The literature on the asymptotics of MDE based on KS distance is very limited, even under parametric assumptions, due to many obstacles (see for example \cite[Section~3.5]{jurevckov:sen:1996} for a nice discussion on the difficulties and challenges). It makes the asymptotic study of \PLFIT{} technically challenging. Uniqueness of minimizers is for example an issue which can not be resolved without additional assumptions, cf.\ \cite{jurevckov:sen:1996}. Within the context of extreme value theory this has also been recognized in \cite{drees:janssen:resnick:wang:2018}.

The main contribution of this article is to prove consistency of \PLFIT{} estimator 
under the semi-parametric assumption stated in \eqref{eq:defn:regvar:tail:dist} on $F$, as summarized in the following theorem. In its statement, we write $F^{\leftarrow}$ for the generalized inverse of $F$, as well as $\probconv$ for convergence in probability:

\begin{thm}[Consistency of \PLFIT{}]
\label{thm:consistency:PLFIT}
Suppose that $\overline{F} \in {\rm RV}_{-\alpha}$, and that $\overline{F}$  and $F^{\leftarrow}$ are continuous. Then $\widehat{\alpha}_{n, \kappa_n^\star} \probconv \alpha$ as $n \to \infty$, so that the \PLFIT{} method is consistent.
\end{thm}

In light of the above discussion, our contribution can be seen as a rare example of an asymptotic result on MDE in a semi-parametric framework. In the next section, we explain how the proof of Theorem \ref{thm:consistency:PLFIT} is organised. Theorem \ref{thm:consistency:PLFIT} is proved in Section \ref{sec-proof-consistency-PLFIT}, for all cases except the eventually Pareto case (where the slowly varying function is constant above a certain value). Theorem \ref{thm:consistency:PLFIT} for the eventually Pareto case is proved in Section \ref{subsec:consistency:hill}.


We stress that we use a completely different approach from the one used in \cite{drees:janssen:resnick:wang:2018} based on Koml\'os-Major-Tusn\'ady approximation. Our approach is mostly influenced and motivated by \cite{mason:1982}. More precisely, we show that $\kappa_n^\star \probconv \infty$ under \eqref{eq:defn:regvar:tail:dist} and $\kappa_n^\star/n \probconv 0$ when the slowly varying function $L$ is not eventually constant. The main challenge and difficulty lies in the proof of these two results. We derive consistency of the \PLFIT{} estimator from the facts that $\kappa_n^\star$ is an intermediate sequence with high probability and $\widehat{\alpha}_{n,k_n} \probconv \alpha$ if $k_n$ is an intermediate sequence, even when it is random.

This paper is organized as follows. A more detailed description of our results and proof strategy is given in Section 2. Section 3 established the consistency of Hill's estimator for random intermediate sequences.
In Section 4, we establish that the cut-off value generated by \PLFIT{} is a random intermediate sequence. The proofs of several technical lemmas can be found in the appendix.




\section{Further results and strategy of proof}
\label{sec-fur-rws-strat}

We assume that we have a random sample $(X_i)_{i=1}^n$ from a distribution function $F$. For  for every $s \in [0,1]$, let $F^\leftarrow(s) = \inf\{ x \in (0,\infty)\colon F(x) \ge s \}$ denote the generalized inverse of $F$. Without loss of generality, we can assume that $F^{\leftarrow}(0) \ge 1$, because we can always add $1$ to each of the observations otherwise (recall that $F$ is the distribution function of a non-negative random variable).  Given the random sample $(X_i)_{i=1}^n$ of size $n$, we compute $D_{n,k}$ (see \eqref{eq:defn:DNK}) for all $ 1 \le k \le n-1$, and we let $\kappa_n^\star = \argmin_{1 \le k \le n} D_{n,k}$. Note that $\kappa_n^\star$ may not be unique and in that case we choose the smallest one among all the available choices. We denote the \PLFIT{} estimate for the index $\alpha$ by $\widehat{\alpha}_{n, \kappa_n^\star} = H_{n,\kappa_n^\star}^{-1}$.

This section is organised as follows. In Section \ref{sec-consistency-intermediate}, we investigate the consistency of $\widehat{\alpha}_{n, k_n}$ for general sequences $(k_n)_{n\geq 1}$ that are possibly {\em random}. In Section \ref{sec-prop-kappan*}, we discuss properties of the optimizer $\kappa_n^\star$ in the \PLFIT{} method. Combining these results, we prove Theorem \ref{thm:consistency:PLFIT} in Section \ref{sec-proof-consistency-PLFIT}.

\subsection{Consistency of the Hill estimator}
\label{sec-consistency-intermediate}
We start by investigating the consistency of Hill's estimator for {\em random} intermediate sequences.
We start by defining what we mean with a random intermediate sequence:

\begin{defn}[Random intermediate sequences]
\label{def-intermediate}
We call a sequence $(k_n)_{n\geq 1}$ a {\em random intermediate sequence} when $k_n\probconv \infty$ and $k_n/n\probconv 0$. In particular, it is allowed that $k_n$ is random and depends on the data.
\end{defn}

A key result in our analysis is the following theorem, that states that $\widehat{\alpha}_{n, k_n}\probconv \alpha$ when $(k_n)_{n\geq 1}$ a random intermediate sequence:

\begin{thm}[Consistency of Hill's estimator for random intermediate sequence]
\label{thm:consistency-Hill-random}
Let $(k_n)_{n\geq 1}$ be a random intermediate sequence. Then Hill's estimator with $k=k_n$ is consistent, i.e., $\widehat{\alpha}_{n, k_n}\probconv \alpha$.
\end{thm}

This result can be seen as an extension of \cite[Proposition 1]{mason:1982}, where consistency of Hill's estimator for
deterministic intermediate sequences is established. Our proof is a refinement of the method of proof in \cite{mason:1982}.
We also refer to \cite{goldie1987slow} for a convergence proof for random intermediate sequences, under stronger assumptions on $F$, and with stronger results.


\subsection{Properties of the \PLFIT{} optimizer $\kappa_n^\star$}
\label{sec-prop-kappan*}
The  most crucial step in the analysis of the \PLFIT{} method is the study of the asymptotic behavior of $ D_{n, \kappa_n^\star} = \min_{1 \le k \le n} D_{n,k}$. Here we start by showing that $D_{n, \kappa_n^\star} \probconv 0$ under the assumption \eqref{eq:defn:regvar:tail:dist}, which has so far not yet been studied in the literature:

\begin{thm}[Minimum KS distance vanishes]
\label{thm:probconv:DKN}
If $\overline{F} \in {\rm RV}_{-\alpha}$, then $D_{n, \kappa_n^\star} \probconv 0$ as $n \to \infty$.
\end{thm}

In Lemma~\ref{lemma:lower:bound:DKN} below, we prove that if $\kappa_n^\star$ does not grow with the sample size, then $D_{n, \kappa_n^\star}$ remains positive for all $n$ with high probability. Combining this fact with Theorem~\ref{thm:probconv:DKN}, we obtain the following theorem:

\begin{thm}[Optimizer of minimum KS distance tends to infinity]
\label{thm-growth-kappan}
If $\overline{F} \in {\rm RV}_{-\alpha}$, then $\kappa_n^\star \probconv \infty$ as $n \to \infty$.
\end{thm}

The next task is to establish that $\kappa_n^\star/n \probconv 0$. For this, we introduce the following definition:

\begin{defn}[Eventually Pareto distributions]
\label{def-eventually-Pareto}
We call a distribution function {\em eventually Pareto} when there exists an $x_0$ and $c>0$ such that $\bar{F}(x)=c x^{-\alpha}$ for all $x\geq x_0$.
Equivalently, this happens when $L$ in \eqref{L-def-RV} is constant above $x_0$. Below, we will assume that $x_0$ is the {\em smallest} value above which $L$ is constant.
\end{defn}

Our main result concerning the \PLFIT{} method is the following theorem. In its statement, we use the notation
	\begin{align}
	{\cal U}(x) = \frac{\alpha \int_0^{\overline{F}(x)} \log F^{\leftarrow}(1- s) \dtv s - (\alpha \log x + 1) \overline{F}(x) }{\int_0^{\overline{F}(x)} \log \frac{F^{\leftarrow}(1- s)}{x} \dtv s }. \label{eq:defn:calu}
	\end{align}

\begin{thm}[\PLFIT{} optimizer $\kappa_n^\star$ is sublinear when not eventually Pareto]
\label{thm-sublinear}
Suppose that $F$ satisfies that $\overline{F} \in {\rm RV}_{-\alpha}$, as well as that $\overline{F}$  and $F^{\leftarrow}$ are continuous. Then,
	\eqn{
	\label{Dn-conv}
	D_{n, [\bar{F}(x)n]} \probconv \sup_{y\geq 1} y^{-\alpha}\Big|y^{{\cal U}(x)}-\frac{L(xy)}{L(x)}\Big|.
	}
Furthermore, for every $x\geq 1$,
	\eqn{
	\label{positive-Dn}
	\sup_{y\geq 1} y^{-\alpha}\Big|y^{{\cal U}(x)}-\frac{L(xy)}{L(x)}\Big|>0,
	}
unless $F$ is eventually Pareto. Consequently, $\kappa_n^\star/n \probconv 0$ as $n \to \infty$ unless $F$ is eventually Pareto.
\end{thm}

We note that ${\cal U}(x)\equiv 0$ for all $x\geq x_0$ in the eventually Pareto case. It turns out that this is the {\em only} way how the right-hand side of \eqref{Dn-conv} can equal zero. Indeed, we claim that when the right-hand side of \eqref{Dn-conv} equals zero, then $F$ is eventually Pareto for some $x_0\leq x$. To see this, we first note that when ${\cal U}(x)\neq 0$, then the right-hand side of \eqref{Dn-conv} is strictly positive, since $y\mapsto L(xy)/L(x)$ is slowly varying, while $y\mapsto y^{{\cal U}(x)}$ is regularly varying with exponent ${\cal U}(x)\neq 0$. Thus, the right-hand side of \eqref{Dn-conv} is strictly positive for all $y$ sufficiently large by Potter's Theorem. When, instead, ${\cal U}(x)=0$, the right-hand side of \eqref{Dn-conv} simplifies to
	\eqn{
	\label{positive-Dn-U=0}
	\sup_{y\geq 1} y^{-\alpha}\Big|1-\frac{L(xy)}{L(x)}\Big|,
	}
which is strictly positive {\em unless} $L(xy)=L(x)$ for all $y\geq 1$. This is equivalent to the statement that $y\mapsto L(y)$ is constant for all $y\geq x$, which (recall Definition \ref{def-eventually-Pareto}) is equivalent to the statement that $F$ is eventually Pareto for some $x_0\leq x$. Despite this simple argument, it is non-trivial to conclude that $\kappa_n^\star/n \probconv 0$ as $n \to \infty$ unless $F$ is eventually Pareto, due to the fact that for \PLFIT{}, we take an {\em infimum} of $D_{n, k}$, and to be able to interchange the limits with the infimum, we are required to prove a tightness result, see Proposition \ref{prop:tightness:suprema} below.

As mentioned in the introduction, the derivation of asymptotic properties of the minimum-distance estimator based on the KS distance is recognized to be a mathematically challenging problem. Apart from the main ideas and structure of proof sketched in this section, working out the details is technically challenging. The aforementioned tightness result (Proposition \ref{prop:tightness:suprema}) is one of the most delicate problems in the proof. Another challenging key property that required considerable effort is the strict positivity of the right hand side of \eqref{Dn-conv}, particularly when an infimum over all $x>F^{\leftarrow}(1-\vep)$ is taken; see Lemma \ref{lemma:sublinear:growth:noteventually:pareto} below.
\smallskip

Unfortunately, our methods do not apply to integer-valued data, but there are reasonable ways to resolve such issues:
\begin{remark}[Integer-valued random variables]
\label{rem-int-rvs}
In many applications, the power-law data is integer-valued. A key example consists of degrees of vertices in real-world networks, an example that has drawn enormous attention (see e.g., \cite{broido2019scale, clauset:shalizi:newman:2009, holme2019rare, voitalov:vanderhoorn:vanderhofstad:krioukov:2018}). Our theory does not apply to this setting, as already observed in \cite{voitalov:vanderhoorn:vanderhofstad:krioukov:2018}. There, it was proposed to replace the data $(X_i)_{i=1}^n$ by $(X_i+U_i)_{i=1}^n$, where $(U_i)_{i=1}^n$ are i.i.d.\ uniform random variables on [0,1].
When $\prob(X=k)>0$ for all $k\geq k_0$ and the distribution function of $X$ is in ${\rm RV}_{-\alpha}$, one can see that the distribution function $F$ of $X+U$ satisfies that $\overline{F} \in {\rm RV}_{-\alpha}$, as well as that $\overline{F}$  and $F^{\leftarrow}$ are continuous.
\end{remark}


\subsection{Consistency of {\rm \PLFIT{}} for non-eventually Pareto distributions: proof of Theorem \ref{thm:consistency:PLFIT}}
\label{sec-proof-consistency-PLFIT}
We now prove Theorem \ref{thm:consistency:PLFIT} when the distribution function $F$ is not eventually Pareto. We recall the subsequence principle (see \cite[Theorem~2.3.2]{durrett:2010}), which connects convergence in probability and almost sure convergence. It says that a sequence of random variables $(X_n)_{n \ge 1}$ converges in probability to $X$ if and only if for every subsequence $(X_{n_m})_{m \ge 1}$, there is a further subsequence $(X_{n_{m_{k}}})_{k\geq 1}$ converging to $X$ almost surely as $k \to \infty$.
To conclude consistency of the \PLFIT{} estimator $\widehat{\alpha}_{n,\kappa_n^\star}$, for any subsequence $(n_{m})_{m \ge 1}$, we thus need to produce a further subsequence $(n_{m_{k}})_{k \ge 1}$ such that $\widehat{\alpha}_{n_{m_{k}}, \kappa^\star_{n_{m_{k}}}}$ converges almost surely to $\alpha$ as $k \uparrow \infty$.  We do this now.

Fix a subsequence $(n_{m})_{m \ge 1}$. As $\kappa_n^\star \probconv \infty$ is proved in Theorem~\ref{thm-sublinear}, and $\kappa_n^\star/n \probconv 0$ is proved in Theorem~\ref{thm-sublinear} under the stated conditions, we can produce a further subsequence $(n_{m_{k}})_{k\geq 1}$ of $(n_{m})_{m\geq 1}$ such that $\kappa_{n_{m_{k}}}^\star\asconv \infty$ and $\kappa_{n_{m_{k}}}^\star/n_{m_{k}}\asconv 0$. As $\kappa_{n_{m_{k}}}^\star$ is an intermediate sequence with probability one, we have that $\widehat{\alpha}_{n_{m_{k}}, \kappa_{n_{m_{k}}}^\star} \probconv \alpha$ as a consequence of consistency of Hill's estimator as in Theorem \ref{thm:consistency-Hill-random}.

Using the subsequence principle once again, we obtain a further subsequence $\tilde{n}_k$ such that $\widehat{\alpha}_{\tilde{n}_k, \kappa^\star_{\tilde{n}_k}}\asconv\alpha$. If we choose $(n_m)_{m\geq 1}$ to be this subsequence, then we obtain consistency of the \PLFIT{} estimator under the assumption stated in Theorem~\ref{thm-sublinear}. These facts prove Theorem  \ref{thm:consistency:PLFIT}, except in the pure or eventual Pareto cases. The latter proofs will be deferred to Section \ref{subsec:consistency:hill}.
\qed

\subsection{Concluding comments}
\label{sec-disc-op}
Though we settle the open question of consistency of the \PLFIT{} estimator in this paper, many open questions remain.
For example, rather than the Kolmogorov-Smirnov distance used in \eqref{eq:defn:DNK}, there are many other functionals of the data that could be minimised.
It would be of interest to investigate the consistency of such methods more generally. The closest to our setting might be the quantile minimization, where  \eqref{eq:defn:DNK} is replaced by
	\begin{align}
	\label{eq:defn:QNK}
	Q_{n, k} := \sup_{u\in[0,1]} \Big|F_k^{\leftarrow}(u)- u^{-1/\widehat{\alpha}_{n, k}}  \Big|,
	\end{align}
where $u\mapsto F_k^{\leftarrow}(u)$ is the empirical distribution of the vector $(X_{n-k + i: n}/X_{n- k : n})_{i=1}^k.$
This method has been investigated with extensive simulations in \cite[display (9) at Page~10]{casperdevries:2019}.
We expect that the method developed in our paper form a promising starting point in establishing whether this procedure is consistent.
Other functionals to be minimized could be the Cram\'er-von Mises functional
$k\int_{0}^{\infty} (F_k(y)-y^{-\widehat{\alpha}_{n, k}})^2dy$, the Anderson-Darling functional whose square equals
$k\int_{0}^{\infty} (F_k(y)-y^{-\widehat{\alpha}_{n, k}})^2 y^{\widehat{\alpha}_{n, k}}/(1-y^{-\widehat{\alpha}_{n, k}})dy$, or the Hellinger distance. Here, the Anderson-Darling criterion puts more mass on the tails of the distribution, so may be more appropriate for power-law distributions.

Apart from consistency, questions about bias and confidence intervals for \PLFIT{} and other similar procedures remain open, we refer to \cite[Section 6.4]{embrechts2013modelling}  and \cite[Section 4.5]{beirlant:goegebeur:segers:teugels:2006} for an overview of the state of the art.
The analysis in \cite{drees:janssen:resnick:wang:2018} suggests that the construction of confidence intervals may be a nontrivial task, as asymptotic normality of the \PLFIT{} estimator does not seem to hold even in the Pareto case. A more comprehensive analysis of \PLFIT{} in this direction would involve a second-order assumption on the slowly varying function $L$, as in \cite{goldie1987slow}. Such an assumption seems required to assess the rate of growth of $\kappa_n^\star$, which would be an essential first step.

\section{Consistency of Hill's estimator and Pareto cases}
\label{sec-consistency-hill-Pareto}
In this section, we investigate the consistency of the Hill estimator for random intermediate sequences, and prove its consistency for the eventual Pareto example.

\subsection{Consistency of Hill's estimator for random intermediate sequences: Proof of Theorem \ref{thm:consistency-Hill-random}}
\label{subsec:consistency:hill}
In this section, we prove Theorem \ref{thm:consistency-Hill-random} by showing that Hilll's estimator is consistent for random intermediate sequences:

\begin{proof}[Proof of Theorem \ref{thm:consistency-Hill-random}]
We start by briefly describing the proof of consistency of $H_{n,k_n}$ given in \cite[Proposition~1]{mason:1982} when $(k_n)_{n\geq 1}$ is an intermediate sequence. Terms appearing in this proof will be used later in other proofs in this article.

Define $U(x) = F^{\leftarrow}(1 - x^{-1}) = (1/\overline{F})^{\leftarrow}(x)\colon [1, \infty) \to [1, \infty)$. It follows from \cite[Proposition~{0.8(v)}]{resnick:1986} that $U \in {\rm RV}_{1/\alpha}$, i.e.\, $U(x) = x^{1/\alpha} \widetilde{L}(x)$ where $\widetilde{L}$ is a slowly varying function. Let $(\unif_{i : n})_{i=1}^n$ be the order statistics in increasing order of their magnitude associated to a random sample $(\unif_{i})_{i=1}^n$ from the ${\rm Uniform}(0,1)$ distribution.  Using Karamata's representation theorem (\cite[Corollary~{2.1}]{resnick:2007}), we obtain
	\begin{align}
	\label{tildeL-def}
	\widetilde{L}(x) = a_1(x) \exp \Big\{ \int_1^x u^{-1} b_1(u) \dtv u\Big\}.
	\end{align}
Here $a_1\colon (0, \infty) \to (0,\infty)$ and $b_1 \colon (0, \infty) \to (0, \infty)$ are bounded and measurable functions satisfying $\lim_{x \to \infty} a_1(x) = a_1$ and $\lim_{x \to \infty} b_1(x) = 0$ (see \cite[discussion after equation (1.3.1') in page~12]{bingham:goldie:teugels:1987}). It can be shown that
	\begin{align}
	H_{n, k_n} & =  \frac{1}{k_n} \sum_{i=1}^{k_n} \Big( \log U(\unif^{-1}_{k_n+ i + 1: n}) - \log U(\unif^{-1}_{k_n + 1: n}) \Big) \nonumber \\
	& = \frac{1}{\alpha} \frac{1}{k_n} \sum_{i=1}^{k_n} \log \frac{\unif_{k_n - i + 1:n}^{-1}}{\unif^{-1}_{k_n + 1:n} } + \frac{1}{k_n} \sum_{i=1}^{k_n} \log \frac{a_1(\unif_{k_n- i + 1:n}^{-1})}{a_1(\unif_{k_n + 1:n}^{-1})} 	\nonumber \\
	& \hspace{1cm} + \frac{1}{k_n} \sum_{i=1}^{k_n} \int_{\unif^{-1}_{k_n + 1: n}}^{\unif_{k_n - i + 1:n}^{-1}} \dtv u~ u^{-1} b_1(u) \nonumber \\
	& =: \TERM{1}_n(k_n) + \TERM{2}_n(k_n) + \TERM{3}_n(k_n). \label{eq:decomposition:hill:estimator}
	\end{align}
Let $(E_i)_{1 \le i \le n}$ be a random sample of size $n$ from the {\rm Exponential} distribution with mean $1$ and the corresponding order statistics will be denoted by $E_{1:n} < E_{2:n} < \ldots < E_{n:n}$. Then we can use the distributional equality
	\begin{align}
	\Big( \log \unif_{1:n}^{-1}, \log \unif_{2:n}^{-1}, \ldots, \log \unif_{n:n}^{-1} \Big) \eqd \Big( E_{n: n}, E_{n-1:n}, \ldots, E_{1:n} \Big)
	\end{align}
with R\'enyi's representation theorem \cite[Theorem~{1.6.1}]{reiss:2012}  to conclude that
	\begin{align}
	\label{term-1-form}
	\TERM{1}_n(k_n) \eqd \frac{1}{\alpha} \Big(\frac{1}{k_n} \sum_{i=1}^{k_n} E_i\Big);
	\end{align}
see also \cite[(7)]{mason:1982}.

Therefore, $\TERM{1}_n \probconv 1/\alpha$ as $k_n \to \infty$. Further, $\TERM{2}_n \probconv 0$ if $k_n \uparrow \infty$ as $a_1$ is a bounded function that converges in zero, and $\TERM{2}_n \probconv 0$ as $b_1$ is a bounded function that converges to zero in zero. These two statements will be made precise below, and extended to {\em random} intermediate sequences. Combining these facts, we obtain that
	\begin{align}
	\label{eq:hill:estimator:bounded:in:probability}
	H_{n, k_n}  \probconv 1/\alpha \qquad \mbox{ if  }k_n \uparrow \infty.
	\end{align}
The above arguments are similar to the proof by Mason \cite{mason:1982}, see in particular \cite[Proof of Proposition 1]{mason:1982}. However, we are dealing with a {\em random} intermediate sequence, which requires some extra care. We will instead prove that, for every $\vep>0$,
	\eqn{
	\label{term-1-analysis}
	\limsup_{K\rightarrow \infty} \limsup_{n\rightarrow \infty} \prob(\sup_{k\geq K} |\TERM{1}_n(k)-1/\alpha|\geq \vep)=0,
	}
while, for $i=2,3$, and for every $\vep>0$,
	\eqn{
	\label{term-23-analysis}
	\limsup_{K\rightarrow \infty} \limsup_{n\rightarrow \infty} \prob(\sup_{K\leq k\leq n/K} |\TERM{i}_n(k)|\geq \vep)=0.
	}
Since, for a random intermediate sequence $(k_n)_{n\geq 1}$, the bounds $K\leq k_n\leq n/K$ hold with high probability for $n$ large, this implies that $H_{n, k_n}\probconv 1/\alpha$, as required.
Thus, we are left to prove \eqref{term-1-analysis} and \eqref{term-23-analysis}.

We start by proving \eqref{term-1-analysis}. For this, we use the union bound together with \eqref{term-1-form} to obtain
	\eqn{
	\prob(\sup_{k\geq K} |\TERM{1}_n(k)-1/\alpha|\geq \vep)\leq \sum_{k=K}^n \prob\Big(\Big|\sum_{i=1}^{k} E_i-k\Big| \geq \vep \alpha k\Big).
	}
A standard large deviations Chernoff bound shows that there exists a $\delta(\vep)>0$ such that	
	\eqn{
	\prob\Big(\Big|\sum_{i=1}^{k} E_i-k\Big| \geq \vep \alpha k\Big)\leq \e^{-\delta(\vep) k}.
	}
This gives that
	\eqn{
	\prob(\sup_{k\geq K} |\TERM{1}_n(k)-1/\alpha|\geq \vep)\leq \sum_{k=K}^n \e^{-\delta(\vep) k}=\e^{-\delta(\vep) K}/(1-\e^{-\delta(\vep)}).
	}
Thus, indeed, \eqref{term-1-analysis} holds.

We continue with \eqref{term-23-analysis} for $i=2$. We note that
	\eqn{
	|\TERM{2}_n(k)|\leq \log{\Big(\frac{\sup_{x\geq \unif_{k:n}^{-1}}a_1(x)}{\inf_{x\geq \unif_{k:n}^{-1}}a_1(x)}\Big)}.
	}
By the above argument and uniformly in $k\leq n/K$,
	\eqn{
	\unif_{k:n}^{-1}\geq \unif_{n/K:n}^{-1} \probconv K,
	}
since $\unif_{nt:n}\probconv t$. Since $\lim_{x \to \infty} a_1(x) = a_1$, for every $\eta>0$, we can find a $K>0$ such that, with probability converging to 1,
	\eqn{
	\log{\Big(\frac{\sup_{x\geq \unif_{k:n}^{-1}}a_1(x)}{\inf_{x\geq \unif_{k:n}^{-1}}a_1(x)}\Big)}\leq \log\Big(\frac{a_1 + \eta}{a_1 - \eta}\Big).
	}
Therefore, \eqref{term-23-analysis} holds for $i=2$.

For \eqref{term-23-analysis} for $i=3$, we note that, almost surely, for all $k\leq n/K$,
	\eqan{
	\TERM{3}_n(k) &\le \Big( \sup_{ u \ge U_{n/K:n}^{-1}} |b_1(u)|  \Big) \frac{1}{k} \sum_{i=1}^{k} \log \frac{\unif_{k - i + 1:n}^{-1}}{\unif^{-1}_{k + 1:n} }\\
	&\le \Big( \sup_{ u \ge U_{n/K:n}^{-1}} |b_1(u)|  \Big) \max_{k\geq K} \frac{1}{\alpha}\Big( \frac{1}{k}\sum_{i=1}^{k} E_i \Big)\nn\\
	&=\Big( \sup_{ u \ge U_{n/K:n}^{-1}} |b_1(u)|  \Big) \max_{k\geq K} |\TERM{1}_n(k)|.\nn
	}
Note that, for every $\varepsilon>0$,
	\eqn{
	\prob\Big( \sup_{u \ge U_{n/K:n}^{-1}}| b_1(u)| \leq  \sup_{u \ge K-1}| b_1(u)|+ \vep \Big) \rightarrow 1.
	}
The quantity $ \sup_{u \ge K-1}| b_1(u)|$
 becomes small when $K$ is large since $b_1(u)\rightarrow 0$ as $u\rightarrow \infty$. Further, $\max_{k\geq K} |\TERM{1}_n(k)|$ is a tight sequence of random variables by the analysis of $\TERM{1}_n(k)$. Therefore, \eqref{term-23-analysis} also holds for $i=3$.
\end{proof}

\subsection{Consistency of \PLFIT{} for eventually Pareto distributions}
\label{subsec:consistency:hill}
In this section, we prove the consistency of the \PLFIT{} method for the eventually Pareto case. For this, we investigate the properties of the (random and data-driven) $\kappa_n^\star$.

We start with the pure Pareto case. Interestingly, in this case, \cite[Corollary 2.2]{drees:janssen:resnick:wang:2018}, see also the discussion below it, suggests that $\kappa_n^\star$ is {\em not} a random intermediate sequence. However, by Lemma \ref{lemma:lower:bound:DKN} and Theorem \ref{thm:probconv:DKN}, it does follow that $\kappa_n^\star\probconv \infty$. For the Pareto distribution, we see that this suffices to prove consistency.

Indeed, consider the split in \eqref{eq:decomposition:hill:estimator}. For the Pareto distribution, $x\mapsto a_1(x)$ is constant, and $b_1(x)\equiv 0$. Thus, $\TERM{2}_n(k)\equiv \TERM{3}_n(k)\equiv 0$ for every $k$. It suffices to investigate $\TERM{1}_n(\kappa_n^\star)$. Since $\kappa_n^\star\probconv \infty$, \eqref{term-1-analysis} shows that $\TERM{1}_n(\kappa_n^\star)\probconv 1/\alpha$, as required. This proves the consistency of the \PLFIT{} method for the Pareto distribution, and thus proves Theorem \ref{thm:consistency:PLFIT} in this case.

We next extend the above proof to the eventually Pareto case. Recall $x_0$ in Definition \ref{def-eventually-Pareto}. In the following proposition, we show that $\kappa_n^\star/n\leq \bar{F}(x_0-\vep)$ with high probability, so that the part of the distribution that is not Pareto is asymptotically avoided:

\begin{propn}[W.h.p.\ $\kappa_n^\star$ avoids the non-Pareto regime]
\label{PROP-EVENT-PARETO}
Assume that $F$ is eventually Pareto above the value $x_0$. Then
	\eqn{
	\prob(\kappa_n^\star\geq n\bar{F}(x_0-\vep))\rightarrow 0.
	}
\end{propn}

The proof of Proposition \ref{PROP-EVENT-PARETO} is deferred to Appendix \ref{app-eventually-Pareto}. It relies on the proof of Theorem \ref{thm-sublinear}.


We now complete the proof of consistency of the \PLFIT{} method in the eventual Pareto case, subject to Proposition \ref{PROP-EVENT-PARETO}:
\smallskip

\noindent
{\it Proof of consistency of the \PLFIT{} method for the eventual Pareto case, subject to Proposition \ref{PROP-EVENT-PARETO}.} We again consider the split in \eqref{eq:decomposition:hill:estimator}. Recall the discussion above \eqref{tildeL-def}. For an eventual Pareto distribution, $x\mapsto a_1(x)$ is constant for $x\geq \bar{F}(x_0)$, and $b_1(x)\equiv 0$ for $x\geq \bar{F}(x_0)$. Thus,
	\eqn{
	\log \frac{a_1(\unif_{k_n- i + 1:n}^{-1})}{a_1(\unif_{k_n + 1:n}^{-1})}=\int_{\unif^{-1}_{k_n + 1: n}}^{\unif_{k_n - i + 1:n}^{-1}} \dtv u~ u^{-1} b_1(u)=0
	}
whenever $\unif_{k_n + 1:n}^{-1}\geq \bar{F}(x_0)$. We wish to apply this argument to $\kappa_n^\star$. By Proposition \ref{PROP-EVENT-PARETO}, any limit point of $\kappa_n^\star/n$ (which is a bounded sequence) is supported on $[0,\bar{F}(x_0)]$. As a result, w.h.p.\ and every $\vep>0$,
	\eqn{
	\log \frac{a_1(\unif_{\kappa_n^\star- i + 1:n}^{-1})}{a_1(\unif_{\kappa_n^\star + 1:n}^{-1})}\leq \log \Big(\frac{\sup_{u\geq \bar{F}(x_0-\vep)}a_1(u)}{\inf_{u\geq \bar{F}(x_0-\vep)}a_1(u)}\Big),
	}
which converges to 1 as $\vep \searrow 0$. A similar argument gives a lower bound that converges to 1 as $\vep \searrow 0$. The same argument applies to the $\int_{\unif^{-1}_{\kappa_n^\star + 1: n}}^{\unif_{\kappa_n^\star - i + 1:n}^{-1}}  u^{-1} b_1(u) \dtv u$ term, so that we conclude that $\TERM{2}_n(\kappa_n^\star), \TERM{3}_n(\kappa_n^\star)\probconv 0$.

It thus again suffices to investigate $\TERM{1}_n(\kappa_n^\star)$. Since $\kappa_n^\star\probconv \infty$, \eqref{term-1-analysis} shows that $\TERM{1}_n(\kappa_n^\star)\probconv 1/\alpha$, as required. This proves the consistency of the \PLFIT{} method for eventually Pareto distributions.
\qed

\section{Proof that $\kappa_n^\star$ is intermediate subject to auxiliary results}
\label{sec-proof-intermediate}
In this section, we reduce the proof of our main results in Theorems~\ref{thm:probconv:DKN}--\ref{thm-sublinear} to three lemmas and three propositions. The proofs of these auxiliary results will be postponed to Appendix \ref{sec-pfs-aux-4} below.

\subsection{Proof of Theorems~\ref{thm:probconv:DKN} and \ref{thm-growth-kappan}}
\label{SUBSEC:KAPPAN:INFTY}
To prove Theorem~\ref{thm:probconv:DKN}, we need to show that $D_{n, \kappa_n^\star} \probconv 0$ under the assumption \eqref{eq:defn:regvar:tail:dist}. While it is tempting to believe that this indeed holds, this fact is not known in the literature  and the mathematical argument behind the answer is far from being obvious. The main ingredients to the proof consist of Proposition \ref{propn:upperbound:distance}, which gives a convenient upper bound on $D_{n, \kappa_n^\star}$ as a sum of three terms, and Lemmas \ref{lemma:first:term:probconv:zero:lower:bound}--\ref{lemma:lower:bound:DKN} that bound these terms. We defer their proofs to Appendix \ref{sec-pfs-aux-4}.

As a first step, we derive a convenient upper bound on the random distance $D_{n,k}$ in the following proposition:

\begin{propn}[Upper bound on KS distance]
\label{propn:upperbound:distance}
For every $k > 1$, almost surely,
	\begin{align}
	\label{eq:upper:bound:DKN}
	D_{n,k} & \le \sup_{y \ge 1} \Big| \frac{1}{k} \sum_{i=1}^k \mbbo_{(y, \infty)} \Big( \frac{X_{n-k+i:n}}{X_{n-k:n}} \Big) - y^{-\alpha} \Big|  \nonumber \\
	& \hspace{.5cm} + \sup_{y \ge 1} \big| y^{- \widehat{\alpha}_{n, k_n}} - y^{- \alpha} \big| + \Big( \frac{X_{n:n}}{X_{n-k:n}} \Big)^{-\widehat{\alpha}_{n,k}}.
	\end{align}
\end{propn}

Given this result, it suffices to show that there exists a sequence $(k_n)_{n\ge 1}$ such that the upper bound in \eqref{eq:upper:bound:DKN} converges to $0$ in probability. This choice of $(k_n)_{n\geq 1}$ may depend on the distribution function $F$. We prove in the following lemma that the first term in \eqref{eq:upper:bound:DKN} converges to $0$ in probability if $(k_n)_{n\geq 1}$ is an intermediate sequence:

\begin{lemma}[KS to limiting Pareto vanishes]
\label{lemma:first:term:probconv:zero:lower:bound}
If $(k_n)_{n \ge 1}$ is an intermediate sequence, then
\begin{align}
\sup_{y \ge 1} \Big| \frac{1}{k} \sum_{i=1}^k \mbbo_{(y, \infty)} \Big( \frac{X_{n-k+i:n}}{X_{n-k:n}} \Big) - y^{-\alpha} \Big|  \probconv 0
\end{align}
as $n \to \infty$ when $\overline{F}$ satisfies \eqref{eq:defn:regvar:tail:dist}.
\end{lemma}

For the second term in the upper bound \eqref{eq:upper:bound:DKN}, it is sufficient to have $\widehat{\alpha}_{n, k_n} \probconv \alpha$:

\begin{lemma}[Consistency of $\widehat{\alpha}_{n, k_n}$ implies vanishing second term]
\label{lemma:second:term:upper:bound:DKN:negligible}
If $\widehat{\alpha}_{n, k_n} \probconv \alpha$, then $\sup_{y \ge 1} |y^{- \widehat{\alpha}_{n,k_n}} - y^{- \alpha}| \probconv 0$ as $n \to \infty$.
\end{lemma}

For the third term in the upper bound \eqref{eq:upper:bound:DKN}, it is necessary and sufficient to have $k_n \uparrow \infty$:

\begin{lemma}[Analysis of ratios of order statistics]
\label{lemma:lower:bound:DKN}
When $\overline{F}$ satisfies \eqref{eq:defn:regvar:tail:dist},
\begin{enumerate}

\item $D_{n, k_n} \ge \e^{- K}$ almost surely when $k_n \le K$; and

\item $(X_{n:n}/ X_{n-k_n : n})^{- \widehat{\alpha}_{n, k_n}}  \probconv 0$ as $n \to \infty$ if and only if $k_n \to \infty$.

\end{enumerate}
\end{lemma}
\smallskip

With these results in hand, we are now ready to complete the proof of Theorem~\ref{thm:probconv:DKN}:

\begin{proof}[Proof of Theorem~\ref{thm:probconv:DKN}]
Choose any intermediate sequence $(k_n)_{n\ge 1}$. By the results of Mason \cite{mason:1982}, $\widehat{\alpha}_{n, k_n} \probconv \alpha$, so that Lemma \ref{lemma:second:term:upper:bound:DKN:negligible} applies. Then $D_{n, \kappa_n^\star} \le D_{n , k_n} \probconv 0$ by Lemmas~\ref{lemma:first:term:probconv:zero:lower:bound}, \ref{lemma:second:term:upper:bound:DKN:negligible} and \ref{lemma:lower:bound:DKN}.
\end{proof}

We further use Theorem~\ref{thm:probconv:DKN} to complete the proof of Theorem~\ref{thm-growth-kappan}:

\begin{proof}[Proof of Theorem~\ref{thm-growth-kappan}]
Note that it is enough to show that $\prob(\kappa_n^\star \le M) \to 0$ as $n \to \infty$ for any $M >1$. By Lemma~\ref{lemma:lower:bound:DKN}, $\prob (\kappa_n^\star \le M) \le \prob( D_{n, \kappa_n^\star} \ge \e^{-M})$. Using Theorem~\ref{thm:probconv:DKN}, we can immediately deduce that $\prob(D_{n, \kappa_n^\star} \ge \e^{-M}) \to 0$ as $n \to \infty$ and hence Theorem~\ref{thm-growth-kappan} follows.
\end{proof}

\subsection{Optimizer $\kappa_n^\star$ does not grow linearly: Proof of Theorem~\ref{thm-sublinear}} \label{SUBSEC:PROOF:SUBLINEAR:GROWTH}
In this section, we prove Theorem~\ref{thm-sublinear}, which is the central part of our proof. In this section, we will reduce this proof to two propositions (Proposition~\ref{propn:dnk:functional:limit} and \ref{prop:tightness:suprema}), and two lemmas (Lemmas~\ref{lem-inf-t[vep,1]} and \ref{lemma:sublinear:growth:noteventually:pareto}), whose proofs are deferred to Section \ref{sec-pfs-aux-42}.

To prove prove Theorem~\ref{thm-sublinear}, we need to show that, for every $\varepsilon > 0$,
	\begin{align}
	\lim_{n \to \infty} \prob(\kappa_n^\star/n > \varepsilon) = 0. \label{eq:aim:upper:bound:kappan}
	\end{align}
Here, $\kappa_n^\star=\argmin_{1 \le k \le n} D_{n, k}$.

\paragraph{Outline of the proof} We start by explaining the outline of the proof. For every $y \ge 1$, we define the sample version of the tail empirical measure (see (4.16) in \cite{resnick:2007}) by
	\begin{align}
	\label{eq:defn:sample:tail:empirical:measure}
    	\widehat{\nu}_{n,k} (y, \infty)  = \frac{1}{k} \sum_{i=1}^n \bdelta_{X_i/X_{n- k : n} } (y, \infty),
	\end{align}
where $ \bdelta$ denotes Kronecker's delta, so that
	\eqn{
	D_{n, k}=\sup_{y\geq 1} \Big|\widehat{\nu}_{n,k} (y, \infty)-y^{-\widehat{\alpha}_{n,k}}\Big|.
	}
Define ${\cal D}_{j:k} = \inf_{ j \le i \le k} D_{n,i}$ for every pair $j<k$ of positive integers. Using this notation, we get that $D_{n, \kappa_n^\star} = \min( {\cal D}_{1:[n\varepsilon] - 1}, {\cal D}_{[n\varepsilon] : n}), $ where $[x]$ denotes the largest integer less than or equal to $x$. For any $\eta >0$, it is immediate that
	$$
	\prob(\kappa_n^\star/n > \varepsilon) \le \prob \Big( \kappa_n^\star/n >\varepsilon, ~  D_{n, \kappa_n^\star} \le \eta \Big) + \prob(D_{n , \kappa_n^\star} > \eta).
	$$
It follows from Theorem~\ref{thm:probconv:DKN} that $\prob(D_{n,\kappa_n^\star} > \eta) = o(1)$. We use the observation that $\{\kappa_n^\star/n > \varepsilon\} \subseteq \{D_{n, \kappa_n^\star} = {\cal D}_{[n\varepsilon]  : n}\}$, where the inclusion follows since $\kappa_n^\star$ is the {\em smallest} minimiser of $k\mapsto D_{n,k}$. This establishes that to prove \eqref{eq:aim:upper:bound:kappan}, it is enough to show that
	\begin{align}
	0&=\lim_{\eta \searrow 0} \lim_{n \to \infty} \prob ( {\cal D}_{[n\varepsilon]: n} \le \eta) \nonumber \\
	& = \lim_{\eta \searrow 0} \lim_{n \to \infty} \prob \Big( \inf_{t \in [\varepsilon,1]} \sup_{y \ge 1} \Big| \widehat{\nu}_{n, [nt]} ([y, \infty)) - y^{-\widehat{\alpha}_{n, [nt]}} \Big| \le \eta \Big) \nonumber \\
	& = \lim_{\eta \searrow 0} \lim_{n \to \infty} \prob \Big( \inf_{t \in [\varepsilon, 1]} \sup_{y \ge 1} |Z_n(t,y)| \le \eta \Big),
	\label{eq:upper:bound:reduced:aim}
	\end{align}
where
	\begin{align}
	\label{eq:defn:Z_n}
	Z_n(t,y) := \widehat{\nu}_{n,[nt]}((y, \infty)) - y^{- \widehat{\alpha}_{n, [nt]}}.
	\end{align}
The remainder of the proof of Theorem~\ref{thm-sublinear} is now organised as follows. In Proposition \ref{propn:dnk:functional:limit}, we show that, for every $t>0$, the process $(Z_n(t,y))_{y\geq 1}$ converges in probability to a {\em deterministic} limiting process. This, in particular, also implies that $\sup_{y\geq 1}|Z_n(t,y)|$ converges, for every fixed $t>0$. Proposition \ref{propn:dnk:functional:limit} extends this convergence to {\em tightness} in $t\in[\vep,1]$ for $t\mapsto \sup_{y\geq 1}|Z_n(t,y)|$. This, in particular, also proves that the infimum over $t\in[\vep,1]$ of $\sup_{y\geq 1}|Z_n(t,y)|$ also converges in probability, as made precise in Lemma \ref{lem-inf-t[vep,1]}. Finally, Lemma \ref{lemma:sublinear:growth:noteventually:pareto} shows that the limiting variable is strictly positive when $F$ is not eventually Pareto (see also the discussion below Theorem \ref{thm-sublinear}). Since this yields a contradiction with Theorem~\ref{thm:probconv:DKN}, it follows that $\kappa_n^\star/n > \varepsilon$ cannot hold with high probability. This then completes the proof of Theorem \ref{thm-sublinear}.
\smallskip

We now start to provide the details in the above outline. Define
	\begin{align}
	Z(t, y) := t^{-1}\overline{F} \big(y F^{\leftarrow}(1-t) \big) - \exp \Big\{- t \Big( \int_0^t \log  \frac{F^{\leftarrow}(1-s)}{ F^{\leftarrow}(1-t)} \dtv s  \Big)^{-1} \log y \Big\}. \label{eq:defn:Zty}
	\end{align}
The main aim of the proof will be to show that
	\begin{align}
	\inf_{t \in [\varepsilon, 1]} \sup_{y \ge 1} |Z_n(t,y)| \probconv \inf_{t \in [\varepsilon, 1]} \sup_{y \ge 1} |Z(t,y)|, \label{eq:prob:limit:Z_n}
	\end{align}
with the limit being strictly positive except in the eventual Pareto case. In the following proposition, we show that $(Z(t,y))_{y\geq 1}$ is the limit in probability of $(Z_n(t,y))_{y \ge 1}$ pointwise in $t \in [\varepsilon,1]$:
%
%
%
%

\begin{propn}[Pointwise convergence of $\sup_{y \ge 1} |Z_n(t, y)|$]
\label{propn:dnk:functional:limit}
Suppose that $\overline{F}$ and $F^{\leftarrow}$ are continuous functions. Then, for $t \in [\varepsilon,1]$,
	\begin{align}
	(Z_n(t,y))_{y \ge 1} \probconv (Z(t,y))_{y \ge 1}
	\end{align}
in the uniform topology on $D[1,\infty)$. As $y\mapsto Z(t,y)$ is continuous in $y$ for $y\geq 1$, it follows that
	\eqn{
	\sup_{y \ge 1} |Z_n(t, y)| \probconv \sup_{y \ge 1} \big|Z(t, y) \big|.
	}
\end{propn}

Define $\overline{Z}_n(t) = \sup_{y \ge 1} |Z_n(t,y)|$ and $\overline{Z}(t) = \sup_{y \ge 1} |Z(t,y)|$ for $t > 0$, so that Proposition \ref{propn:dnk:functional:limit} shows that $\overline{Z}_n(t)\probconv \overline{Z}(t)$ pointwise in $t\in [\vep,1]$. Note, however, that pointwise convergence in $t$ is not enough to conclude convergence of the process $(\overline{Z}_n(t))_{t \in [\vep,1]}$ and we additionally need tightness. In the next proposition, we prove tightness of the process $(\overline{Z}_n(t))_{t \in [\varepsilon,1]}$:

\begin{propn}[Tightness of $\overline{Z}_n(t) = \sup_{y \ge 1} |Z_n(t,y)|$]
\label{prop:tightness:suprema}
The process $\overline{Z}_n(t) = \sup_{y \ge 1} |Z_n(t,y)|$ satisfies that, almost surely and uniformly in $t \in [\varepsilon,1]$,
	\begin{align}
	\sup_{|h| \le \delta}  \big| \overline{Z}_n(t+h) - \overline{Z}_n(t) \big| \le f(\delta)
	\end{align}
where $f$ satisfies that $\lim_{\delta \to 0} f(\delta) = 0$. Consequently, $(\overline{Z}_n(t))_{t\in [\vep,1]}$ is a tight sequence of stochastic processes.
\end{propn}
\smallskip

By Propositions \ref{propn:dnk:functional:limit}-\ref{prop:tightness:suprema}, we obtain \eqref{eq:prob:limit:Z_n} using tightness and finite-dimensional convergence, as stated in the next lemma:

\begin{lemma}[Convergence in probability of ${\cal D}_{[n\varepsilon]  : n}$]
\label{lem-inf-t[vep,1]}
If $\overline{F}$ and $F^\leftarrow$ are continuous functions, then
	\begin{align}
	\inf_{t \in [\varepsilon,1]} \sup_{y \ge 1} |Z_n(t,y)| \probconv \inf_{t \in [\varepsilon,1]} \sup_{y \ge 1} |Z(t,y)|.
	\end{align}
Consequently, also ${\cal D}_{[n\varepsilon]  : n}\probconv \inf_{t \in [\varepsilon,1]} \sup_{y \ge 1} |Z(t,y)|$.
\end{lemma}
\smallskip

It follows from Propositions~\ref{propn:dnk:functional:limit} and \ref{prop:tightness:suprema}, and  Lemma~\ref{lem-inf-t[vep,1]}, that
	\begin{align}
	\label{eq:lim_n:prob}
	\prob ( {\cal D}_{[n\varepsilon]: n} \le \eta) \rightarrow \prob \Big( \inf_{t \in [\vep,1]} \sup_{y \ge 1}|Z(t,y)| \le \eta \Big).
	\end{align}

It is clear that $ \inf_{t \in [\varepsilon, 1]}\sup_{y \ge 1}|Z(t,y)|$ is a deterministic function. Therefore, the probability in \eqref{eq:upper:bound:reduced:aim} actually equals either 0 or 1 depending on whether the number $ \inf_{t \in [\varepsilon, 1]}\sup_{y \ge 1}|Z(t,y)|$ is positive or zero. In the next lemma, we show that the limit is positive unless $F$ is eventually Pareto:


\begin{lemma}[Positivity of the limiting infimum except when $F$ is eventually Pareto]
\label{lemma:sublinear:growth:noteventually:pareto}
$\inf_{t \in [\varepsilon,1]} \sup_{y \ge 1}|Z(t,y)| > 0$ for every $\varepsilon > 0$ unless $F$ is eventually Pareto.
\end{lemma}
\smallskip

The proof of Lemma \ref{lemma:sublinear:growth:noteventually:pareto} is challenging since $\overline{Z}(t) = \sup_{y \ge 1} |Z(t,y)|$ may not be a continuous function of $t$, but is a lower semicontinuous function, instead. We are now ready to complete the proof of Theorem~\ref{thm-sublinear}:

\begin{proof}[Proof of Theorem~\ref{thm-sublinear}]
Assume that $F$ is not eventually Pareto. Then $\inf_{t \in [\varepsilon,1]} \overline{Z}(t)>0$ by Lemma \ref{lemma:sublinear:growth:noteventually:pareto}. Therefore, for all $\eta \in (0, \inf_{t \in [\varepsilon,1]} \overline{Z}(t))$, the probability in \eqref{eq:lim_n:prob} equals $0$. Hence the limit in \eqref{eq:upper:bound:reduced:aim} equals zero, and the proof of Theorem~\ref{thm-sublinear} follows.
\end{proof}


\paragraph{\bf Acknowledgement}
The authors thank Parthanil Roy for helpful discussions. AB acknowledges the support provided by {\sc{EURANDOM}} for his visit to Eindhoven University of Technology during November 17-23, 2019.

\bibliographystyle{plain}

\appendix

\section{Proof of auxiliary results used in Section~\ref{SUBSEC:KAPPAN:INFTY}}
\label{sec-pfs-aux-4}
In this section, we prove the auxiliary results from Section~\ref{SUBSEC:KAPPAN:INFTY}, namely, Proposition~\ref{propn:upperbound:distance} and Lemmas~\ref{lemma:first:term:probconv:zero:lower:bound}--\ref{lemma:lower:bound:DKN}, in that order.

\begin{proof}[Proof of Proposition~\ref{propn:upperbound:distance}]
Let $\nu_\alpha$ be the sigma-finite measure on $((0, \infty), \mathscr{B}(0, \infty))$ such that $\nu_\alpha (y, \infty) = y^{-\alpha}$ for every $y > 0$. 
With this notation, and recalling $\widehat{\nu}_{n,k} (y, \infty)$ in \eqref{eq:defn:sample:tail:empirical:measure}, we obtain that, almost surely,
	\begin{align}
 	D_{n,k} & = \sup_{y \ge 1} \Big| \widehat{\nu}_{n, k} (y, \infty) - \nu_{\widehat{\alpha}_{n, k}}(y, \infty) \Big| \nonumber \\
    	& = \max \Big( \sup_{1 \le y \le X_{n:n} / X_{n-k: n}} \Big| \widehat{\nu}_{n, k} (y, \infty) - \nu_{\widehat{\alpha}_{n, k}}(y, \infty) \Big|, \nonumber \\
    	& \hspace{1cm} \sup_{y > {X_{n:n}}/{X_{n-k: n}}} \Big| \widehat{\nu}_{n, k} (y, \infty) - y^{\widehat{\alpha}_{n, k}} \Big| \Big). \label{eq:tail:empirical:upper:bound}
	\end{align}
Note that $\widehat{\nu}_{n, k} (X_{n:n}/ X_{n- k : n}, \infty) = 0$, which implies that
	\begin{align}
	\sup_{y > {X_{n:n}}/{X_{n-k: n}}} \Big| \widehat{\nu}_{n, k} (y, \infty) - y^{\widehat{\alpha}_{n, k}} \Big| = \Big( \frac{X_{n:n}}{ X_{n-k : n}} \Big)^{- \widehat{\alpha}_{n, k}}.
	\label{eq:tail:empirical:upper:bound:term2}
	\end{align}
Using that $\max(a, b) \leq a + b$ for any $a, b >0$, and combining \eqref{eq:tail:empirical:upper:bound} and \eqref{eq:tail:empirical:upper:bound:term2}, we obtain
	\begin{align}
	D_{n, k_n} & \le  \sup_{ 1 \le y \le X_{n:n}/ X_{n-k: n}} \big| \widehat{\nu}_{n, k}(y, \infty) - \nu_{\widehat{\alpha}_{n, k}}(y, \infty) \big| + \Big( \frac{X_{n:n}}{X_{n-k:n}} \Big)^{- \widehat{\alpha}_{n, k}} \nonumber \\
    	& =: \Iterm{1} + \Iterm{2}.  \label{eq:upper:bound:DKN:ineq:one}
	\end{align}
Using the triangle inequality, we obtain
	\begin{align}
	\Iterm{1}  & \le \sup_{y \ge 1} \Big| \widehat{\nu}_{n, k}(y, \infty) - \nu_\alpha(y, \infty) \Big| + \sup_{1 \le y \le X_{n: n}/ X_{n-k : n}} \Big| y^{- \widehat{\alpha}_{n, k}} - y^{-\alpha} \Big| \nonumber \\
    	& : = \Iterm{11} + \Iterm{12}. \label{eq:upper:bound:DKN:ineq:two}
	\end{align}
Also note that
	\begin{align}
	\Iterm{12} & =   \sup_{1 \le y \le X_{n:n}/X_{n-k:n}} \big| \exp \big\{ - \alpha \log y\} - \exp \big\{ - \widehat{\alpha}_{n,k} \log y \big\} \big| \nonumber \\
	&  \le \big| \widehat{\alpha}_{n, k} - \alpha \big| \log \big(X_{n: n}/X_{n-k:n} \big),
	\label{eq:upper:bound:DKN:ineq:three}
	\end{align}
since $|\e^{-x} - \e^{-y}| \le |x - y|$ for all $x, y \ge 0$. Thus, the proof follows from \eqref{eq:upper:bound:DKN:ineq:one}, \eqref{eq:upper:bound:DKN:ineq:two} and \eqref{eq:upper:bound:DKN:ineq:three}.
\end{proof}

\begin{proof}[Proof of Lemma~\ref{lemma:first:term:probconv:zero:lower:bound}]
Let $\mathscr{M}_+$ denote the space of all locally finite point measures $\mu$ on $(0, \infty)$, that is, $\mu(B) < \infty$ for all $B \in \mathscr{B}((0, \infty))$ and $0 \notin {\rm cl}(B)$ where ${\rm cl}(B)$ denotes the closure of the set $B$.

We say a sequence $(\mu_n)_{n \ge 1}$ converges {\em vaguely} to a measure $\mu$ in $\mathscr{M}_+$ (and write $\mu_n \vconv \mu$) if $\int f \dtv \mu_n \to \int f \dtv \mu $ for all bounded and continuous functions $f$ that vanishes in a neighbourhood of $0$.   It follows from \cite[Theorem~{4.2 ({\sc Step} 2)}]{resnick:2007} that $\widehat{\nu}_{n, k_n} \vconv \nu_\alpha$ in probability if $(k_n)_{n \ge 1}$ is an intermediate sequence.

Therefore, given a subsequence $(n_m)_{m \ge 1}$, there exists a further subsequence $(n_{m_l})_{l \ge 1}$ such that $\widehat{\nu}_{n_{m_l}, k_{n_{m_l}}} \vconv \nu_\alpha$ almost surely. As $[y, \infty)$ is a compact subset of $(0, \infty)$ for all $y \ge 1$, it follows from vague convergence (see \cite[Theorem~{3.2(b)}]{resnick:2007}) that $\widehat{\nu}_{n_{m_l}, k_{_{n_{m_l}}}}([y, \infty)) \to \nu_\alpha([y, \infty))$ almost surely for every $y \ge 1$. Thus, we can use the Glivenko-Cantelli theorem (\cite[Theorem~{19.1}]{vandervaart:2000}) to conclude that $\sup_{y \ge 1} |\nu_{n_{m_l}, k_{n_{m_l}}}([y, \infty)) - \nu_\alpha([y, \infty))| \to 0$ almost surely. We can use subsequential characterization of convergence in probability (see \cite[Theorem~2.3.2]{durrett:2010}) once again to conclude the proof.
\end{proof}

%

\begin{proof}[Proof of Lemma~\ref{lemma:second:term:upper:bound:DKN:negligible}]
Fix $\vep > 0$ and $0 < \eta < \min( \alpha/2, \log(\vep + 1)/ \log 2)$. Then we immediately have the decomposition
	\begin{align}
	& \prob \Big( \sup_{y \ge 1} \Big| y^{- \widehat{\alpha}_{n, k_n}} - y^{- \alpha}  \Big| > \vep \Big)  \nonumber \\
	& \le \prob \Big( \sup_{y \ge 1} \Big| y^{- \widehat{\alpha}_{n, k_n}} - y^{- \alpha} \Big| > \vep, ~|\widehat{\alpha}_{n, k_n} - \widehat{\alpha}| \le  \eta \Big)
	+ \prob \Big( |\widehat{\alpha}_{n, k_n} - \alpha| > \eta \Big).
	\label{eq:upper:bound:sup:polynomial}
	\end{align}
The second term in \eqref{eq:upper:bound:sup:polynomial} vanishes as $n \to \infty$ as a consequence of $\widehat{\alpha}_{n, k_n} \probconv \alpha$. Therefore, it suffices to show that the first term in \eqref{eq:upper:bound:sup:polynomial} vanishes.
\smallskip

Throughout the following line of reasoning, we assume that the event $\{|\widehat{\alpha}_{n, k_n} - \alpha| < \eta\}$ occurs.
Observe that
	\begin{align*}
	\sup_{y \ge 1} |y^{- \widehat{\alpha}_{n,k_n}} - y^{- \alpha}|
	&= \sup_{y \ge 1} \max \big( y^{- \widehat{\alpha}_{n, k_n}} - y^{- \alpha}, y^{- \alpha} - y^{- \widehat{\alpha}_{n, k_n}} \big)\\
	&\le \sup_{y \ge 1} \max ( y^{- \alpha + \eta} - y^{- \alpha}, y^{- \alpha} - y^{- \alpha - \eta}).
	\end{align*}
In addition,  $\max(y^{- \alpha + \eta} - y^{- \alpha}, y^{- \alpha} - y^{- \alpha - \eta}) = y^{- \alpha + \eta} - y^{- \alpha}$. Combining these two estimates leads to
	\begin{align}
	\sup_{y \ge 1} \big| y^{- \widehat{\alpha}_{n, k_n}} - y^{- \alpha} \big| \le \sup_{y \ge 1} (y^{- \alpha + \eta} - y^{- \alpha}) = \sup_{y \ge 0} \Big(\e^{(- \alpha + \eta)  y} - \e^{- \alpha y} \Big).
	\end{align}
The supremum on the r.h.s.\ is attained at $y = \log (\alpha/ (\alpha - \eta))$,  so that
	\begin{align}
	\sup_{y \ge 0} \Big(\e^{(- \alpha + \eta)  y} - \e^{- \alpha  y} \Big) \le \Big( \frac{\alpha}{\alpha- \eta} \Big)^{- \alpha} \Big[ \Big(\frac{\alpha}{\alpha - \eta} \Big)^{- \eta}  - 1 \Big] \le 2^{\eta} - 1,
	\end{align}
using the facts that $\alpha/(\alpha - \eta) > 1$ and $\eta < \alpha /2$. We have also assumed that $\eta < \log ((\vep + 1)/2)$ implying $2^{\eta} - 1 < \vep$. We conclude that the first probability appearing in \eqref{eq:upper:bound:sup:polynomial} equals zero, completing the proof.
\end{proof}

\begin{proof}[Proof of Lemma~\ref{lemma:lower:bound:DKN}]
Let $k_n \le K$ for all $n \ge 1$. Then,
	\begin{align}
	\Big( \frac{X_{n:n}}{X_{n-k_n:n}} \Big)^{- \widehat{\alpha}_{n, k_n}} = \exp \Big\{ - H_{n,k_n}^{\inv} \log \frac{X_{n:n}}{X_{n-k_n:n}} \Big\} \ge \e^{-K}, \label{eq:finite:kn:lowerbound}
	\end{align}
since trivially
	\eqn{
	\label{tightness-alpha_n}
	H_{n,k_n} \ge \frac{1}{K} \log \frac{X_{n:n}}{X_{n-k_n:n}}.
	}
Therefore, the first claim in Lemma~\ref{lemma:lower:bound:DKN} follows from \eqref{eq:tail:empirical:upper:bound} and \eqref{eq:tail:empirical:upper:bound:term2} combined with \eqref{eq:finite:kn:lowerbound}.

For the second claim, we note that $\big( \log(X_{n:n}/ X_{n-k_n:n}) \big)^{-1} $ is $o_{\sss \prob}(1)$ if and only if $k_n \uparrow \infty$.
Further, $H_{n,k_n}=O_{\sss \prob}(1)$ when $k_n \uparrow \infty$ by Theorem \ref{thm:consistency-Hill-random} (which even shows that $H_{n,k_n}\probconv 1/\alpha<\infty$).
Combining both estimates, we conclude that, when $k_n \uparrow \infty$,
	\begin{align}
	\Big( \log \frac{X_{n:n}}{X_{n-k_n:n}} \Big)^{-1} H_{n, k_n} \probconv 0.
	\end{align}
When $k_n\leq K$, on the other hand, $\widehat{\alpha}_{n, k_n}=1/H_{n,k_n}$ remains tight by \eqref{tightness-alpha_n}, and $X_{n:n}/X_{n-k_n:n}\leq X_{n:n}/X_{n-K:n}$ also remains tight, so that
$(X_{n:n}/ X_{n-k_n : n})^{- \widehat{\alpha}_{n, k_n}}$ does not converge to zero in probability. This completes the proof of the second claim in Lemma~\ref{lemma:lower:bound:DKN}.
\end{proof}

\section{Proofs of the auxiliary results in Section~\ref{SUBSEC:PROOF:SUBLINEAR:GROWTH}}
\label{sec-pfs-aux-42}
In this section, we prove the auxiliary results in Section~\ref{SUBSEC:PROOF:SUBLINEAR:GROWTH}, in the order Proposition~\ref{propn:dnk:functional:limit}, Proposition~\ref{prop:tightness:suprema}, Lemma~\ref{lem-inf-t[vep,1]} and Lemma~\ref{lemma:sublinear:growth:noteventually:pareto}.

\subsection{Proof of Proposition~\ref{propn:dnk:functional:limit}}




We start with the almost sure limit of the Hill estimator and this will be used as the key ingredient in the following proof and also in proof of Proposition~\ref{prop:tightness:suprema}.

\begin{lemma}[Almost sure limit of $H_{n,[nt]}$]
\label{lemma:aslimit:hill}
Under the assumptions stated in Proposition~\ref{propn:dnk:functional:limit}, for every $t \in [\varepsilon,1]$,
\begin{align}
H_{n,[nt]} \asconv H_t:=\frac{1}{t} \int_0^t \log \frac{F^{\leftarrow}( 1-s)}{F^{\leftarrow}(1 - t)} \dtv s.\label{eq:aim:flt:hill}
\end{align}
\end{lemma}
\smallskip

We now start proving Proposition~\ref{propn:dnk:functional:limit}, and begin by giving an outline to reduce the proof to the two claims in \eqref{eq:quantile:empirical:conv} and \eqref{eq:functional:limit:hill} below.
The proof will then be completed by proving \eqref{eq:quantile:empirical:conv} and \eqref{eq:functional:limit:hill}.

We need to show that, as $n \to \infty$,
	\begin{align}
	\sup_{y \ge 1} \big|Z_n(t,y) - Z(t,y) \big| \probconv 0.
	\end{align}
Let $F_n$ denote the empirical distribution function, that is,
	\eqn{
	\label{Fn-def}
	F_n(x) = n^{-1} \sum_{i=1}^n \mbbo(X_i \le x),
	}
and $\overline{F}_n = 1 - F_n$ the empirical tail distribution. We then first observe that
	\begin{align*}
	\widehat{\nu}_{n, [nt]}(y, \infty) & = \frac{1}{[nt]} \sum_{i =1}^{[nt]} \mbbo_{(y, \infty)} \Big( \frac{X_{n - [nt] + i: n}}{X_{n -[nt]:n }} \Big)\\
	&  = \frac{1}{[nt]} \sum_{i =1}^{n} \mbbo(X_i > y X_{n - [nt]:n}) = \frac{n}{[nt]} \overline{F}_n(y X_{n - [nt]: n}).
	\end{align*}
The proof of Proposition~\ref{propn:dnk:functional:limit} follows by combining \eqref{eq:defn:Z_n} and \eqref{eq:defn:Zty} with the following two convergence claims:
	\begin{align}
	& \sup_{y \ge 1} \Big| \frac{n}{[nt]} \overline{F}_n(y X_{n - [nt]:n}) - t^{-1} \overline{F}(y F^{\leftarrow}(1-t)) \Big|  \asconv 0, \label{eq:quantile:empirical:conv} \\
	\mbox{and }\qquad& \sup_{y \ge 1} \Big| \exp \{- \widehat{\alpha}_{n,[nt]} \log y\} - \exp \Big\{-\frac{t}{  \int_0^t \log \frac{F^{\leftarrow}(1 - s)}{F^{\leftarrow}(1 - t)} \dtv s }  \log y \Big\} \Big| \asconv 0.  	
	\label{eq:functional:limit:hill}
	\end{align}
Indeed, by \eqref{eq:defn:Z_n} and \eqref{eq:defn:Zty}, $\sup_{y \ge 1} \big|Z_n(t,y) - Z(t,y) \big|$ is bounded from above by the sum of the two right hand sides in \eqref{eq:quantile:empirical:conv} and \eqref{eq:functional:limit:hill}. The rest of the proof is dedicated to the proofs of \eqref{eq:quantile:empirical:conv} and \eqref{eq:functional:limit:hill}.
\smallskip

\noindent {\it Proof of \eqref{eq:quantile:empirical:conv}}. By the triangle inequality, the convergence claims
	\begin{align}
	& \sup_{y \ge 1} \frac{n}{[nt]} \Big|  \overline{F}_n(y X_{n - [nt]:n}) -  \overline{F}_n \big( y F^{\leftarrow}(1-t) \big) \Big| \asconv 0, \label{eq:quantile:empirical:conv:one} \\
	\mbox{and }\qquad & \sup_{y \ge 1} \Big| \frac{n}{[nt]} \overline{F}_n(y F^{\leftarrow}(1-t)) - t^{-1} \overline{F}(y F^{\leftarrow}(1-t)) \Big| \asconv 0. \label{eq:quantile:empirical:conv:two}
	\end{align}
yield \eqref{eq:quantile:empirical:conv}. These two convergence claims will be obtained separately.
\smallskip

We start with \eqref{eq:quantile:empirical:conv:one}. We can ignore the term $n/[nt]$ as it is bounded for all $t \in [\varepsilon,1]$. Fix $\delta \in (0, F^{\leftarrow}(1 - \varepsilon))$. The continuity of $F^{\leftarrow}$ is assumed and so we can use the Glivenko-Cantelli theorem for the quantile process  (see \cite[(1.4.7)]{csorgho:1983}) to conclude that $X_{n- [nt]:n} \asconv F^{\leftarrow}(1 - t)$. Therefore, there exists an event $B$ satisfying $\prob(B) = 1$, such that $|X_{n - [nt]:n} - F^{\leftarrow}(1 - t)| < \delta$ for all $n \ge N_1$ on the event $B$, where $N_1$ is an appropriately chosen large integer. On the event $B$ and for $n \ge N_1$,
	\begin{align}
	& \Big| \overline{F}_n \big( y X_{n - [nt]:n} \big) - \overline{F}_n \big( y F^{\leftarrow}( 1- t) \big) \Big| \nonumber \\
	& = \max \Big( \overline{F}_n(y X_{n - [nt] : n}) - \overline{F}_n(y F^{\leftarrow}(1 - t)), \overline{F}_n(y F^{\leftarrow}( 1- t)) - \overline{F}_n \big( y X_{n - [nt]: n} \big) \Big) \nonumber \\
	& \le \overline{F}_n(y F^{\leftarrow}( 1- t) - y\delta)  - \overline{F}_n(y F^{\leftarrow}(1- t) + y \delta) \nonumber \\
	& \le \Big| \overline{F}_n(y F^{\leftarrow}( 1- t) - y\delta) - \overline{F}(y F^{\leftarrow}( 1- t) - y \delta) \Big| \nonumber \\
	& \hspace{1cm}  + \Big| \overline{F}_n(y F^{\leftarrow}( 1- t) + y\delta) - \overline{F}(y F^{\leftarrow}( 1- t) + y \delta) \Big|  \nonumber \\
	& \hspace{2cm} +  \Big( \overline{F}\big( y F^{\leftarrow}( 1 - t) - y \delta \big) - \overline{F}(y F^{\leftarrow}( 1-t) + y \delta) \Big) \nonumber \\
	& = {\rm I}_{n, \delta}^{\sss(1)} + {\rm I}_{n, \delta}^{\sss(2)} + {\rm I}_\delta^{\sss(3)}. \label{eq:decomposition:empirical:quantile:conv:one}
	\end{align}
We investigate each of these three terms, starting with ${\rm I}_\delta^{\sss(3)}$. Fix a large number ${\rm K} > 1$. Write
	\[
	{\rm I}_\delta^{\sss(3)}(y)=
	\overline{F}\big( y F^{\leftarrow}( 1 - t) - y \delta \big) - \overline{F}(y F^{\leftarrow}( 1-t) + y \delta).
	\]
Then,
	\[
	{\rm I}_\delta^{\sss(3)}=\sup_{y \ge 1} {\rm I}_\delta^{\sss(3)}(y)  \le \sup_{ 1 \le y \le {\rm K}} {\rm I}_\delta^{\sss(3)}(y) + \sup_{ y \ge {\rm K}} {\rm I}_\delta^{\sss(3)}(y).
	\]
It follows from uniform continuity of $y\mapsto {\rm I }_\delta^{\sss(3)}(y)$ that $\lim_{\delta \to 0} \sup_{1 \le y \le {\rm K}} {\rm I}_\delta^{\sss(3)}(y) = 0$ for every fixed ${\rm K} \ge 1$. Note that ${\rm I}_\delta^{\sss(3)}(y) \le 2 \overline{F}(y F^{\leftarrow}( 1- t) - y \delta)$ and so $\sup_{y \ge {\rm K}} {\rm I}_\delta^{\sss(3)}(y) \le \overline{F}({\rm K}(F^{\leftarrow}( 1 - t) - \delta))$. It is easy to see that
	\begin{align*}
	\lim_{{\rm K} \to \infty} \limsup_{\delta \to 0} \sup_{y \ge {\rm K}} {\rm I}_\delta^{\sss(3)}(y) \le \lim_{{\rm K} \to \infty} \lim_{\delta \to 0} \overline{F} \Big( {\rm K}( F^{\leftarrow}( 1- t) - \delta) \Big) = 0,
	\end{align*}
and therefore $\lim_{\delta \to 0} {\rm I}_\delta^{\sss(3)} = 0$.

By the Glivenko-Cantelli Theorem, $\sup_{y \ge 1} {\rm I}_{n,\delta}^{\sss(i)} \asconv 0$ for $i=1,2$ and every $\delta > 0$. Combining these facts and the decomposition derived in \eqref{eq:decomposition:empirical:quantile:conv:one}, we conclude that \eqref{eq:quantile:empirical:conv:one} holds. The proof of \eqref{eq:quantile:empirical:conv:two} again follows from the Glivenko-Cantelli Theorem and the fact $n/[nt] \to 1/t$ as $n \to \infty$. This completes the proof of \eqref{eq:quantile:empirical:conv}.\\

\noindent{\it Proof of \eqref{eq:functional:limit:hill}}. Note that the only $n$-dependence in \eqref{eq:functional:limit:hill} is in $\widehat{\alpha}_{n,[nt]}$. Further, $\widehat{\alpha}_{n, [nt]} = H_{n, [nt]}^{-1}$. We first complete the proof of \eqref{eq:functional:limit:hill} subject to Lemma \ref{lemma:aslimit:hill}, and then complete the proof by proving Lemma \ref{lemma:aslimit:hill}. By Lemma \ref{lemma:aslimit:hill}, \eqref{eq:functional:limit:hill} reduces to
	\begin{align}
	\sup_{y \ge 1} |y^{- \widehat{\alpha}_{n, [nt]}} - y^{- H_t^{-1}}| \asconv 0.  \label{eq:reduced:functional:limit:hill}
	\end{align}
By the intermediate value theorem, for some $\beta_{n,y}$ in between $1/H_{n, [nt]}$ and $1/H_t$,
	\begin{align}
 	& \sup_{y \ge 1} \big| y^{- H_{n, [nt]}^{-1} } - y^{-H_t^{-1}} \big| \leq |H_{n, [nt]}^{-1}-H_t^{-1}|
	\sup_{y \ge 1} (\log{y})y^{-\beta_{n,y}}.
	\end{align}
Since $H_{n,[nt]} \asconv H_t>0$, we have that $\sup_{y \ge 1} (\log{y}) y^{-\beta_{n,y}}$ is uniformly bounded for all $n$ such that $1/H_{n, [nt]}\geq 1/[2H_t]$, while $|H_{n, [nt]}^{-1}-H_t^{-1}|\asconv 0$. This completes the proof of \eqref{eq:functional:limit:hill}.
\qed
\smallskip
We complete the proof of Proposition~\ref{propn:dnk:functional:limit} by proving Lemma~\ref{lemma:aslimit:hill}:

\begin{proof}[Proof of Lemma~\ref{lemma:aslimit:hill}]  Let $(\unif_i)_{i \ge 1}$ denote a collection of independent ${\rm Uniform}(0,1)$ random variables. The order statistics of the first $n$ elements of $(\unif_i)_{i \ge 1}$ are denoted by $\unif_{1:n} < \unif_{2:n} < \cdots < \unif_{n:n}$. We denote the empirical distribution of $(\unif_i)_{1 \le i \le n}$ by $F_n^{\sss (\unif)}$. We start by observing that
	\begin{align}
	H_{n, [nt]} \eqd \frac{1}{[nt]} \sum_{i=1}^{[nt]} \Big[\log F^{\leftarrow}(1 - \unif_{i:n}) - \log F^{\leftarrow}( 1 - \unif_{[nt]:n})\Big].
	\label{eq:hill:uniform:representation}
	\end{align}
In view of \eqref{eq:hill:uniform:representation}, the proof of Lemma~\ref{lemma:aslimit:hill} reduces to proving
	\begin{align}
	& \log F^{\leftarrow} \big( 1 - \unif_{[nt]:n} \big) \asconv \log F^{\leftarrow}( 1- t),  \label{eq:aslimit:hill:termone} \\
	\mbox{ and } \qquad&\frac{nt}{[nt]} \frac{1}{t} \Big( \frac{1}{n} \sum_{i =1}^{[nt]} \log F^{\leftarrow}( 1- \unif_{i:n}) \Big) \asconv \frac{1}{t} \int_0^t \log F^{\leftarrow}( 1- s) \dtv s. \label{eq:aslimit:hill:termtwo}
	\end{align}
Note that $\unif_{[nt]:n} \asconv t$ and by the continuous mapping theorem, we get \eqref{eq:aslimit:hill:termone} as $F^{\leftarrow}$ is assumed to be continuous. Therefore, the rest of the proof is dedicated to the proof of \eqref{eq:aslimit:hill:termtwo}.

We first note that $nt/[nt] \to 1$. Therefore, \eqref{eq:aslimit:hill:termtwo} reduces to proving
	\begin{align}
	\frac{1}{n} \sum_{i=1}^{[nt]} \log F^{\leftarrow}( 1 - \unif_{i : n}) \asconv \int_0^t \log F^{\leftarrow}(1 - s) \dtv s. \label{eq:aslimit:hill:termtwo:reduced}
	\end{align}
We start by decomposing the difference of the left and right hand sides of \eqref{eq:aslimit:hill:termtwo:reduced} as
	\begin{align}
	& \Big| \frac{1}{n} \sum_{i=1}^{[nt]} \log F^{\leftarrow}( 1 - \unif_{i : n}) - \int_0^t \log F^{\leftarrow}(1 - s) \dtv s \Big| \nn \\
	&  \le \Big| \frac{1}{n}\sum_{i=1}^{[nt]} \log F^{\leftarrow}( 1 - \unif_{i:n}) - \frac{1}{n} \sum_{i=1}^n \mbbo(\unif_i \le t) \log F^{\leftarrow}( 1- \unif_i) \Big| \nn \\
	& \hspace{1cm} + \Big| \frac{1}{n} \sum_{i=1}^n \mbbo(\unif_i \le t) \log F^{\leftarrow}(1 - \unif_i) - \int_0^t \log F^{\leftarrow}(1 - s) \dtv s  \Big|\nn\\
	& =: \Iterm{1} + \Iterm{2}. \label{eq:aslimit:hill:decomposition}
\end{align}
Our task is now reduced to showing that
	\begin{align}
	\Iterm{1} \asconv 0 \qquad \mbox{ and }\qquad \Iterm{2} \asconv 0.
	\end{align}

\noindent {\it Proof that $\Iterm{1} \asconv 0$}. We know that $\unif_{n, [nt]} \asconv t$ as $n \to \infty$. Fix $\vep \in (0, t/2)$. Then, there exists an event $B$ satisfying $\prob(B) = 1$ and a large (random) integer $N$ such that $|\unif_{n, [nt]} - t| < \vep$ on the event $B$ for all $n \ge N$. Therefore, for all $n \ge N$, on the event $B$,
	\begin{align}
	\Iterm{1} & \le \frac{1}{n} \sum_{i =1}^n \Big( \mbbo(t \le \unif_i \le \unif_{[nt]:n}) + \mbbo(\unif_{[nt]:n} \le \unif_i \le t) \Big) \log F^{\leftarrow}( 1- \unif_i) \nn \\
	& \le \Big( \log F^{\leftarrow}( 1- t+ \vep) \Big) \Big( \frac{1}{n} \sum_{i =1}^n \mbbo( t - \vep \le \unif_i \le t + \vep) \Big) \nn\\
	& = \log F^{\leftarrow}( 1- t + \vep) \Big( F_n^{\sss (\unif)}(t + \vep) - F_n^{\sss (\unif)}(t - \vep) \Big) \nn\\
	& \asconv 2 \vep \Big( \log F^{\leftarrow}( 1- t + \vep) \Big), \label{eq:upper:bound:aslimit:hill}
	\end{align}
by the Glivenko-Cantelli Theorem for the empirical process $F_n^{\sss (\unif)}$. Finally, we can see that the limit obtained in the right hand side of \eqref{eq:upper:bound:aslimit:hill} converges to zero if we let $\vep \to 0$. This completes the proof of the fact that $\Iterm{1} \asconv 0$.\\

\noindent{\it Proof of $\Iterm{2} \asconv 0$}. We first observe that
	\begin{align*}
	\frac{1}{n} \sum_{i =1}^n \mbbo(\unif_i \le t) \log F^{\leftarrow}(1 - \unif_i) = \int_0^t \log F^{\leftarrow}( 1- u) F_n^{\sss (\unif)}(\dtv u).
	\end{align*}
Therefore, the strong law of large numbers implies that $\Iterm{2} \asconv 0$ when we can show that
	\begin{align}
	\exptn \Big| \mbbo(\unif \le t) \log F^{\leftarrow}( 1- \unif) \Big| < \infty. \label{eq:finite:expectation}
	\end{align}
Our next task will be to verify \eqref{eq:finite:expectation}. Note that $F^{\leftarrow}(1 - x^{-1}) \in {\rm RV}_{1/\alpha}$ at infinity (see \cite[Proposition~0.8(iv)]{resnick:1986}). This means that $F^{\leftarrow}( 1 - x) \in {\rm RV}_{-1/\alpha}$  at $0$, that is, $F^{\leftarrow}(1 - x) = x^{- 1/\alpha} L_*(x)$ for all $x \in (0,1]$, where $L_*$ is a slowly varying function. Fix $\rho \in (0, 1/\alpha)$. Then we can choose $x_1$ small enough  such that $L_*(x) \le x^{-\rho} $ for all $x \le x_1$ (see \cite[Proposition~0.8(ii)]{resnick:1986}). Using these bounds, we obtain
	\begin{align*}
	& \exptn\Big| \mbbo(\unif \le t) \log F^{\leftarrow}(1 - \unif) \Big| \nn \\
	& =  \int_0^t |\log F^{\leftarrow}(1 - x)| \dtv x \nn \\
	& \le \int_{x_1}^t |\log (x^{1/\alpha} L_*(x)) \dtv x| + (\alpha^{-1} + \rho) (x_1 \log x_1^{-1} + x_1) < \infty.
	\end{align*}
Hence the proof of Lemma~\ref{lemma:aslimit:hill} is complete.
\end{proof}

\subsection{Proof of Proposition~\ref{prop:tightness:suprema}}

In this section, we prove Proposition~\ref{prop:tightness:suprema}. To explain our proof strategy, we begin by making some first estimates, using the triangle inequality.
Recall that $ t \in [\varepsilon, 1]$. Choose $\delta < \varepsilon/3$. We start the analysis with the observation that
	\begin{align}
	Z_n(t,y)= \max \Big[ \sup_{ 1\le y \le X_{n:n}/X_{n- [nt]:n}} \Big| \widehat{\nu}_{n, [nt]}(y, \infty) - y^{- \widehat{\alpha}_{n, [nt]}} \Big|, \Big(\frac{X_{n:n}}{X_{n - [nt]:n}} \Big)^{- \widehat{\alpha}_{n, [nt]}} \Big].
	\end{align}
Recall that $\overline{Z}_n(t)=\sup_{y\geq 1}|Z_n(t,y)|$. It is straightforward to see that there exists at least one $y_0 \in [1, \infty)$ such that $\overline{Z}_n(t) = Z_n(t, y_0)$ and $y_h \in [1, \infty)$ such that $\overline{Z}_n(t+h) = Z_n(y+h, y_h)$ for every $h \in (- \delta, \delta)$. Note that, almost surely,
	\begin{align}
	& |\overline{Z}_n(t) - \overline{Z}_n(t + h)| \nn\\
	&= \max \Big( Z_n(t, y_0) - Z_n(t+h, y_h), Z_n(t+h, y_h) - Z_n(t, y_0) \Big) \nonumber \\
	& \le  \max \Big( Z_n(t,y_0) - Z_{n}(t, y_0), Z_n(t+h, y_h) - Z_n(t, y_h) \Big) \nn\\
	& \le \sup_{y \ge 1} |Z_n(t,y) - Z_n(t+h, y)| \nn\\
	& \le \sup_{y \ge 1} |\widehat{\nu}_{n, [nt]}(y, \infty) - \widehat{\nu}_{n, [n(t+h)]}(y, \infty)| + \sup_{y \ge 1} |y^{- \widehat{\alpha}_{n, [nt]}} - y^{- \widehat{\alpha}_{n, [n(t + h)]}}| \nn \\
	& =: \Iterm{1} + \Iterm{2}.
	\end{align}
Furthermore, note that, almost surely,
	\begin{align}
	\mathrm{I}_n^{\sss(2)} & \le \sup_{ |h| < \delta} \sup_{y \ge 1} \big| y^{- \widehat{\alpha}_{n, [n(t + h)]}} - y^{- H_{t + h}^{-1}} \big|
	+ \sup_{y \ge 1} \big| y^{- \widehat{\alpha}_{n , [nt]}} - y^{- H_t^{- 1}} \big| \nonumber \\
	& \hspace{1cm} + \sup_{|h| < \delta } \sup_{y \ge 1} \big| y^{- H_t^{-1}} - y^{- H_{t + h}^{-1}} \big| \nonumber \\
	& = : \mathrm{I}_n^{\sss(21)} + \mathrm{I}_n^{\sss(22)} + \mathrm{I}^{\sss(23)}.
	\end{align}

Our proof strategy consists of developing suitable upper bounds for each of the above four terms. To bound  $\mathrm{I}_n^{\sss(22)}$, we apply Lemma \ref{lemma:aslimit:hill}. To handle
the terms  $\mathrm{I}_n^{\sss(1)}$  and $\mathrm{I}_n^{\sss(21)}$, we state two auxiliary results below.
In particular, the term $\mathrm{I}_n^{\sss(1)}$ is handled by the following lemma, which shows that the process $(F_n(y X_{n - [nt]:n}))_{t \in [\varepsilon,1], y \ge 1}$ is tight in the $J_1$-topology in ${\bf D}[\varepsilon,1] \times [1, \infty)$.

\begin{lemma}[Tightness of empirical process evaluated at order statistic] \label{lemma:In11:tightness}
Under the assumptions stated in Theorem~\ref{thm-sublinear}, there exists a large number $N_1$ such that  for all $n \ge N_1$, $t \in [\varepsilon,1]$ $\delta \in (0, \varepsilon/3)$, almost surely,
	\begin{align}
	\label{equation:lemma:In11:tightness}
	\sup_{|h| < \delta} \sup_{y \ge 1} \Big| F_n \big( y X_{n -[nt]:n} \big) - F_n \big( y X_{n - [n(t + h)] : n} \big) \Big| \le 26 \delta.
	\end{align}
\end{lemma}

To estimate the term $\mathrm{I}_n^{\sss(21)}$, we use the following functional limit theorem for the inverse of Hill's estimator $(H_{n, [nt]})_{t \in [\varepsilon,1]}$, which may be of independent interest:

\begin{propn}[Functional limit theorem for inverse Hill estimator] \label{propn:flt:hills:estimator}
If $\overline{F}$ and $F^{\leftarrow}$ is continuous, for every $\vep > 0$,
	\begin{align}
	(H_{n, [nt]})_{t \in [\vep, 1]} \asconv \Big( t^{-1} \int_0^t \log \frac{F^{\leftarrow}(1- s)}{F^{\leftarrow}(1 - t)} \dtv s\Big)_{t \in [\vep,1]}.
	\end{align}
\end{propn}

Before we prove these results, we apply them to prove Proposition~\ref{prop:tightness:suprema}:\\

\noindent {\bf Step 1: upper bound for $\Iterm{1}$}. Recall that $\widehat{\nu}_{n, [nt]}(y, \infty) = n \overline{F}_n(y X_{n- [nt]:n})$, and write
	\begin{align}
	\sup_{|h| < \delta} \Iterm{1} & = \sup_{|h| < \delta} \sup_{y \ge 1}\Big| \frac{n}{[n(t + h)]} \overline{F}_n \big( y X_{n - [n(t + h)]: n} \big) - \frac{n}{[nt]} \overline{F}_n(y X_{n- [nt]:n}) \Big| \nn \\
	& \le \frac{1}{t - \delta}  \sup_{ |h|< \delta} \sup_{y \ge 1} \frac{n}{[n(t+ h)]}   \Big| F_n \big( y X_{n -[nt]:n} \big) - F_n \big( y X_{n - [n(t + h)] : n} \big) \Big| \nn \\
 	&  \hspace{2cm} + \sup_{|h| < \delta} \big| \frac{n}{[nt]} - \frac{n}{[n(t + h)]} \big| \nn\\
	& =: {\rm I}_n^{\sss(11)} + {\rm I}_n^{\sss(12)}, \label{eq:decomposition:Inone}
	\end{align}
almost surely. Using basic algebra, we obtain that
	\begin{align}
	{\rm I}_n^{\sss(12)} \le \frac{2\delta}{t(t - 2 \delta)} \le 6 \delta \varepsilon^{-2}, \label{eq:upper:bound:IN:12}
	\end{align}
for all $n \ge N_2$ if $N _2$ is chosen large enough.

Combining Lemma~\ref{lemma:In11:tightness} and \eqref{eq:upper:bound:IN:12}, for all $n \ge \max(N_1, N_2)$,
	\begin{align}
	\Iterm{1} \le 26 \delta +  6 \delta \varepsilon^{-2}.  \label{eq:bound:tightness:Inone}
	\end{align}

\noindent {\bf Step 2: upper bound for ${\rm I}^{\sss(23)}$}. 
It is clear that
	\begin{align}
	H_{t + h}^{-1} - H_t^{-1}  & = (t +h) \Big[ \int_0^{t + h} \log \frac{F^{\leftarrow}(1 - s)}{F^{\leftarrow}( 1 - t - h)} \dtv s \Big]^{-1}
	- t \Big[\int_0^t \log \frac{F^{\leftarrow}( 1 - s)}{F^{\leftarrow}( 1 - t)} \dtv s \Big]^{-1}\nonumber \\
	& = {\rm C}_1(t,h) \Big( - t \int_t^{t + h} \log \frac{F^{\leftarrow}(1 - s)}{F^{\leftarrow}(1 - t - h)} \dtv s  + h \int_0^t \log \frac{F^{\leftarrow}(1 - s)}{F^{\leftarrow}(1 - t)} \dtv s  \nonumber \\
	&  \hspace{2cm} + t^2 \log \frac{F^{\leftarrow}( 1 - t - h)}{F^{\leftarrow}(1 - t)} \Big), \label{eq:difference:positive:H}
	\end{align}
where
	\begin{align}
	{\rm C}_1(t, h) = \Big( \int_0^t \log \frac{ F^{\leftarrow}( 1- s)}{F^{\leftarrow}(1 - t)} \dtv s \Big)^{-1} \Big( \int_0^{t + h} \log \frac{F^{\leftarrow}( 1 - s)}{ F^{\leftarrow}(1-t - h)} \dtv s \Big)^{-1} > 0.
	\end{align}
Similarly, one can derive that
	\begin{align}
	H_{t - h}^{-1} - H_t^{-1} & = {\rm C}_1(t-h, h)  \Big[ t \int_{t - h}^t \log \frac{F^{\leftarrow}(1 - s)}{F^{\leftarrow}(1 - t)} \dtv s - h \int_0^t \log \frac{F^{\leftarrow}( 1 - s)}{F^{\leftarrow}( 1 - t )} \dtv s \nonumber \\
	& \hspace{2.5cm} + t^2 \log \frac{F^{\leftarrow}( 1- t + h)}{F^{\leftarrow}( 1 - t)}  \Big]. \label{eq:difference:negative:H}
	\end{align}
Since $t\geq \vep$, and when $|h|\leq \delta\leq \vep/3$,
	\begin{align}
	{\rm C}_1(t, h) \le \Big( \int_0^{\varepsilon/2} \log \frac{F^{\leftarrow}(1 - s)}{F^{\leftarrow}(1 - \varepsilon/2)} \dtv s \Big)^{-2} =:{\rm C}_2(\varepsilon/2).
	\end{align}
Let $g(\delta)$ be the modulus of continuity of the uniformly continuous function $\log F^{\leftarrow}(1 -t) : [\varepsilon, 1] \to (1, \infty)$ and set $\norm{\log F^{\leftarrow}}_1 = \int_0^1 \log F^{\leftarrow}(1 - s) \dtv s$, where we use that $F(1)=0$ so that $F^{\leftarrow}(1 - s)>1$.

We  use the triangle inequality with \eqref{eq:difference:positive:H} and \eqref{eq:difference:negative:H} to see that
	\begin{align}
	\sup_{|h| < \delta}|H_t^{-1} - H_{t + h}^{-1}| \le 2 g(\delta) + \delta \norm{\log F^{\leftarrow}}_1.
	\end{align}
We proceed by noting that
	\begin{align}
	\mathrm{I}^{\sss(23)} & = \sup_{|h| < \delta} \sup_{y \ge 1} \max \Big( y^{- H_t^{-1}} - y^{- H_{t + h}^{-1}}, y^{-H_t^{-1}} - y^{- H_{t+ h}^{-1}} \Big) \nonumber \\
	& \le \sup_{y \ge 1} y^{- H_t^{-1} + 2 g(\delta) + \delta \norm{\log F^{\leftarrow}}_1} \big( 1 - y^{ - 2 g(\delta) - \delta \norm{\log F^{\leftarrow}}_1} \big) \nonumber \\
	& \le \sup_{ 1 \le y \le {\rm C}_3} y^{- H_t^{-1} + 2 g(\delta) + \delta \norm{\log F^{\leftarrow}}_1} \big( 1 - y^{-2 g(\delta) - \delta \norm{\log F^{\leftarrow}}_1} \big) \nonumber \\
	& \hspace{1cm} + \sup_{y \ge {\rm C}_3} y^{- H_t^{-1} + 2 g(\delta) + \delta \norm{\log F^{\leftarrow}}_1} \nonumber \\
	& \le \Big( 2 g(\delta) + \delta \norm{ \log F^{\leftarrow}}_1 \Big) \log {\rm C}_3 + C_3^{ - H_t^{-1} + 2 g(\delta) + \delta \norm{\log F^{\leftarrow}}_1} \nonumber \\
	& \le \delta  + \Big( 2 g(\delta) + \delta \norm{ \log F^{\leftarrow}}_1 \Big) \log {\rm C}_3, \label{eq:upper:bound:IN:23}
	\end{align}
if ${\rm C_3} > \exp\{(H_t^{-1} - 2 g(\delta) - \delta \norm{\log F^{\leftarrow}}_1) \log \delta^{-1}\}$ and $\delta$ is chosen small enough so that $H_t^{-1} - 2 g(\delta) - \delta \norm{\log F^{\leftarrow}}_1 > 0$.\\

\noindent {\bf Step 3: upper bound for ${\rm I}_n^{\sss(21)}$.}
 Define $\overline{H}_\varepsilon := \sup_{t \in [\varepsilon,1]} H_t > 0$. Note that $\overline{H}_\varepsilon$ is well defined as $t \mapsto H_t$ is a bounded continuous function when $t \in [\varepsilon, 1]$. Consider a positive integer ${\rm C}_4 > \delta^{ -  \overline{H}^{-1}_{\varepsilon/2}}$. Fix $0 < \delta_3 < \min(\overline{H}_{\varepsilon/ 2}, (\delta/ \log {\rm C}_4))$. We can use Proposition~\ref{propn:flt:hills:estimator}  to establish the existence of an integer $N^{\sss(3)}$ such that
	\begin{align}
	\sup_{|h| < \delta} \big| H^{-1}_{n, [n(t+ h)]} - H^{-1}_{t + h} \big| < \delta_3
	\end{align}
almost surely for all $n \ge N^{\sss(3)}$ as $H_{t + h} > 0$ for all $h \in [- \delta, \delta]$.
Applying this estimate twice, we obtain, almost surely for all $n \ge N^{\sss(3)}$,
\begin{align}
\mathrm{I}_n^{\sss(21)} & \le \sup_{|h| < \delta} \sup_{y \ge 1} \max \Big( y^{- \widehat{\alpha}_{n, [n(t+ h)]}} - y^{- H_{t +h}^{-1}}, y^{- H_{t + h}^{-1}} - y^{- \widehat{\alpha}_{n, [n(t+h)]}} \Big) \nonumber \\
& \le \sup_{|h| < \delta}\sup_{y \ge 1} y^{- H_{t + h}^{-1} + \delta_3} \big( 1 - y^{- \delta_3} \big) \nonumber \\
& \le \sup_{1 \le y \le {\rm C}_4}  \big( 1 - y^{- \delta_3} \big) + \sup_{|h| < \delta}\sup_{y \ge {\rm C}_4} y^{- H_{t + h}^{-1} + \delta_3}  \nonumber \\
& \le \delta_3 \log {\rm C}_4 + {\rm C}_4^{- \overline{H}_{\varepsilon/2}^{-1} + \delta_3} <2\delta. \label{eq:upper:bound:IN:21}
\end{align}

\noindent{\bf Final step: putting the pieces together.}
Lemma~\ref{lemma:aslimit:hill} implies that $\mathrm{I}_n^{\sss(22)} \asconv 0$. Therefore, there exists an integer $N^{\sss(1)}$ such that  $\mathrm{I}_n^{\sss(22)} < \delta$ almost surely for $n \ge N^{\sss(1)}$.
  Combining \eqref{eq:upper:bound:IN:21}, \eqref{eq:upper:bound:IN:23}  and $\mathrm{I}_n^{\sss(22)} < \delta$, for all $n \ge N^{\sss(1)} \vee N^{\sss(3)}$, we obtain that
\begin{align}
\mathrm{I}_n^{\sss(2)} \le 4 \delta + \big(2 g(\delta)  + \delta \norm{\log F^{\leftarrow}}_1 \big) \log {\rm C}_3 \mbox{ almost surely.}  \label{eq:upperbound:IN:2}
\end{align}
We conclude Proposition \ref{propn:flt:hills:estimator} from \eqref{eq:bound:tightness:Inone} and \eqref{eq:upperbound:IN:2}.
\qed

\smallskip

We conclude this section by proving the auxiliary results used in the proof of Proposition \ref{propn:flt:hills:estimator}, i.e., Lemma~\ref{lemma:In11:tightness} and Proposition~\ref{propn:flt:hills:estimator}:

\begin{proof}[Proof of Lemma~\ref{lemma:In11:tightness}]
Here we first decompose the l.h.s.\ of (\ref{equation:lemma:In11:tightness})
into two terms depending on whether $h$ is positive or negative, to obtain
	\begin{align}
	& \sup_{|h| < \delta} \sup_{y \ge 1} \Big| \overline{F}_n(y X_{n- [nt]:n}) - \overline{F}_n(y X_{n - [n(t + h)] : n}) \Big|  \nn\\
	& \le \frac{1}{t - \delta} \Big[ \sup_{0 < h < \delta} \sup_{y \ge 1}  \Big| \overline{F}_n(y X_{n- [nt]:n}) - \overline{F}_n(y X_{n - [n(t + h)] : n}) \Big| \nonumber \\
	& \hspace{1cm} + \sup_{0 < h < \delta} \sup_{y \ge 1} \Big| \overline{F}_n(y X_{n- [nt]:n}) - \overline{F}_n(y X_{n - [n(t + h)] : n}) \Big|  \Big] \nonumber \\
	& := \frac{1}{t - \delta} (\mathrm{I}_n^{\sss(1)} + \mathrm{I}_n^{\sss(2)}).
	\end{align}
We derive an upper bound for $\mathrm{I}_n^{\sss(1)}$; the upper bound for $\Iterm{2}$ follows along the same lines. We decompose $\Iterm{1}$ further as
 %
%
	\begin{align}
	\Iterm{1} & \le  \sup_{y \ge 1} \big| \overline{F}_n(y X_{n - [nt]:n}) - \overline{F}_n(y F^{\leftarrow}( 1 - t)) \big| \nn \\
	& \hspace{1cm} + \sup_{0 < h < \delta} \sup_{y \ge 1} \big| \overline{F}_n \big( y X_{n- [n(t + h)] : n} \big) - \overline{F}_n(y F^{\leftarrow}( 1- t - h)) \big|  \nn\\
	& \hspace{2cm} + \sup_{0 < h < 1} \sup_{y \ge 1} \big| \overline{F}_n(y F^{\leftarrow}(1 - t)) - \overline{F}_n(y F^{\leftarrow}( 1- t - h)) \big| \nn\\
	& =: {\rm I}_n^{\sss(11)} + {\rm I}_n^{\sss(12)} + {\rm I}_n^{\sss(13)}. \label{eq:first:decomposition:first:term}
	\end{align}
In the remainder of the proof, we deal with each term separately.
%
%

\noindent{\bf Step 1: term ${\rm I}_n^{\sss(11)}$}. We have assumed that $F^{\leftarrow}(0) \ge 1$ and so we can choose a large real number ${\rm C}_{\delta}$ such that
	\begin{align}
	\overline{F}({\rm C}_\delta F^{\leftarrow}(0)) < \delta/2. \label{eq:vanishing:tail}
	\end{align}
Since $F$ is uniformly continuous, for every $\delta>0$, there exists an $\eta$ such that
	\begin{align}
	|F(x) - F(y)|< \delta ~~~\mbox{ if } |x - y| < \eta.  \label{eq:uniform:continuity:distn:function}
	\end{align}
We choose $\delta_1$ such that $0 < \delta_1 < \eta/{\rm C}_\delta$.  Thus, there exists a large integer $N_{\delta_1}$ such that
	\begin{align}
	\sup_{0 < h < \delta} |X_{n - [n(t + h)] : n} - F^{\leftarrow}(1 - t - h)| < \delta_1 \label{eq:quantile:process:conv}
	\end{align}
for all $n \ge N_{\delta_1}$. Consequently, we obtain
	\begin{align*}
	& \overline{F}_n \big( y X_{n - [nt]: n}\big) - \overline{F}_n (y  F^{\leftarrow}(1 - t))  \le \overline{F}_n \big( y (F^\leftarrow(1 - t) - \delta_1) \big) - \overline{F}_n(y F^{\leftarrow}(1 - t))   \\
	& \mbox{ and } \overline{F}_n (y F^{\leftarrow}( 1- t) \big) - \overline{F}_n \big( y X_{n - [nt]: n} \big) \le \overline{F}_n(y F^{\leftarrow}(1 - t) \big)  - \overline{F}_n \big( y F^{\leftarrow}(1 - t) + y \delta_1 \big).
	\end{align*}
Combining these two inequalities, for $n \ge N_{\delta_1}$, we have following upper bound for $\mathrm{I}_n^{\sss(11)}$:
	\begin{align}
 	&  \sup_{y \ge 1} \max \Big( \overline{F}_n (y X_{n - [nt]:n}) - \overline{F}_n(y F^{\leftarrow}(1 - t)),  \overline{F}_n(y F^{\leftarrow}(1-t)) - \overline{F}_n(y X_{n - [nt]:n}) \Big) \nn \\
	& \le \sup_{y \ge 1} \max \Big(\overline{F}_n \big( y (F^\leftarrow(1 - t) - \delta_1) \big) - \overline{F}_n(y F^{\leftarrow}(1 - t)), \nn \\
	& \hspace{2cm}  \overline{F}_n(y F^{\leftarrow}(1 - t) \big)  - \overline{F}_n \big( y F^{\leftarrow}(1 - t) + y \delta_1 \big) \Big) \nn\\
	& \le \sup_{y \ge 1} \big| \overline{F}_n \big( y (F^\leftarrow(1 - t) - \delta_1) \big) - \overline{F}_n(y F^{\leftarrow}(1 - t)) - \overline{F}(y F^{\leftarrow}(1- t) - y \delta_1) \nonumber \\
	& \hspace{1cm}+ \overline{F}(y F^{\leftarrow}(1-t)) \big|  + \sup_{y \ge 1} \big| \overline{F}_n \big( y (F^\leftarrow(1 - t) + \delta_1) \big) - \overline{F}_n(y F^{\leftarrow}(1 - t)) \nonumber \\
	& \hspace{1cm} - \overline{F}(y F^{\leftarrow}(1- t) + y \delta_1) + \overline{F}(y F^{\leftarrow}(1-t)) \big| + \sup_{y \ge 1} \big| \overline{F}(y \finv(1 - t)) \nn\\
	&  \hspace{1.5cm} - \overline{F}( y \finv(1-t) - y \delta_1) \big| + \sup_{y \ge 1}  \big| \overline{F}(y \finv(1 - t)) \nn \\
	& \hspace{1.75cm} - \overline{F}(y \finv(1 - t) + y \delta_1) \big| =: {\rm I}_n^{\sss(111)} + {\rm I}_n^{\sss(112)} + {\rm I}^{\sss(113)} + {\rm I}^{\sss(114)}. \label{eq:decomposition:In11}
	\end{align}
Using the Glivenko-Cantelli Theorem for the empirical process, we see that there exists a large integer $N_1$ such that
	\begin{align}
	{\rm I}_n^{\sss(111)} + {\rm I}_n^{\sss(112)} \le \delta
	\end{align}
almost surely for all $n \ge N_1$. Note that ${\rm I}^{\sss(113)}$ and ${\rm I}^{\sss(114)}$ can be treated in the same way. We therefore only focus on the latter and note that 
	\begin{align}
	\mathrm{I}^{\sss(114)} & \le \sup_{1 \le y \le {\rm C}_\delta} \Big[ F \Big( y F^{\leftarrow}(1 - t) + y \delta_1 \Big) - F\Big( y F^{\leftarrow}(1 - t) \Big)  \Big] \nonumber \\
	& \hspace{.5cm} +  \sup_{y \ge {\rm C}_\delta} \Big[ \overline{F} \Big( y F^{\leftarrow}(1 - t) \Big) - \overline{F} \Big( y F^{\leftarrow}(1 - t) + y \delta_1 \Big) \Big] \nonumber \\
	& := \mathrm{I}^{\sss(1141)} + {\rm I}^{\sss(1142)}.
	\end{align}
We can see that $\mathrm{I}^{\sss(1141)} < \delta$ by the choice of $\delta_1$ and ${\rm C}_ \delta$, made in \eqref{eq:uniform:continuity:distn:function} and \eqref{eq:vanishing:tail} respectively. It is easy to see that ${\rm I}^{\sss(1142)} < 2 \overline{F}({\rm C}_\delta F^{\leftarrow}(1 - t)) < \delta$ due to the choice of ${\rm C}_\delta$ in \eqref{eq:vanishing:tail}.  Therefore,
	\begin{align}
	{\rm I}^{\sss(114)} \le 2\delta.
	\end{align}
Combining all these facts, we obtain that
	\begin{align}
	{\rm I}_n^{\sss(11)} < 5\delta \mbox{ almost surely for all } n \ge N_1. \label{eq:upperbound:IN:1111}
	\end{align}

\noindent{\bf Step 2: term ${\rm I}_n^{\sss(12)}$}.  The following almost sure inequalities hold due to the quantile process convergence stated in \eqref{eq:quantile:process:conv}:
	\begin{align*}
	& F_n \Big( y X_{n - [n(t + h)]:n}, y F^{\leftarrow}(1 - t - h) \Big) \le F_n \Big( y F^{\leftarrow}(1 - t - h) - y \delta_1, y F^{\leftarrow}(1 - t - h) \Big), \\
	& F_n \Big( y F^{\leftarrow}(1 - t - h), y X_{n - [n(t + h)]:n}  \Big) \le F_n \Big( y F^{\leftarrow}(1 - t - h), y F^{\leftarrow}(1 - t - h) + y \delta_1 \Big).
	\end{align*}
We use these inequalities to derive the almost sure upper bound for ${\rm I}_n^{\sss(12)}$
	\begin{align}
	\mathrm{I}_n^{\sss(12)} & \le \sup_{0 < h < \delta} \sup_{y \ge 1} \Big| F_n \Big( y F^{\leftarrow}( 1 -t - h) - y \delta_1, y F^{\leftarrow}(1 - t - h) \Big)  \nonumber \\
	& \hspace{.25cm} - F\big(y F^{\leftarrow}(1 - t - h) - y\delta_1\big) + F \big( y F^{\leftarrow}(1 - t - h) \big) \Big| \nonumber \\
	& \hspace{.5cm} + \sup_{0 < h < \delta} \sup_{y \ge 1} \Big[ F \Big( y F^{\leftarrow}(1 - t - h)  \Big)  - F \Big( y F^{\leftarrow}(1 - t - h) - y \delta_1 \Big)\Big] \nonumber \\
	& \hspace{.75cm} + \sup_{0 < h < \delta} \sup_{y \ge 1} \Big| F_n \Big( y F^{\leftarrow}(1 - t - h), y F^{\leftarrow}(1 - t - h) + y \delta_1 \Big) \nonumber \\
	& \hspace{1cm} - F \Big( y F^{\leftarrow}(1 - t - h) \Big) + F^{\leftarrow}\Big( y F^{\leftarrow}( 1 - t - h) + y \delta_1 \Big) \Big| \nonumber \\
	&  \hspace{1.25cm} + \sup_{0 < h < \delta} \sup_{y \ge 1} \Big[ F \Big( y F^{\leftarrow}( 1 - t - h) + y \delta_1 \Big) - F \Big( y F^{\leftarrow}(1 - t - h) \Big) \Big] \nonumber \\
	& := {\rm I}_n^{\sss(121)} + {\rm I}^{\sss(122)} + {\rm I}_n^{\sss(123)} + {\rm I}^{\sss(124)}.
	\end{align}
We omitted the details because these are same as in the derivation of \eqref{eq:decomposition:In11}.
The Glivenko-Cantelli Theorem  implies that there exists a large integer $N^{\sss(2)}$ such that
	\begin{align}
	{\rm I}_n^{\sss(121)} + {\rm I}_n^{\sss(123)} < \delta
	\end{align}
for all $n \ge N^{\sss(2)}$.

The term ${\rm I}^{\sss(122)}$ and term ${\rm I}^{\sss(124)}$ can be treated in the exact same way. We therefore focus only on the latter term. Note that
	\begin{align}
	{\rm I}^{\sss(124)} & \le \sup_{0 < h < \delta} \sup_{1 \le y \le {\rm C}_\delta}  \Big[ F \Big( y F^{\leftarrow}( 1- t - h) + y \delta_1 \Big) - F \Big( y F^{\leftarrow}(1 - t - h) \Big) \Big]  \nonumber \\
	& \hspace{1cm} + \sup_{0 < h < \delta} \sup_{y \ge {\rm C}_\delta}  \Big[ F \Big( y F^{\leftarrow}( 1- t - h) + y \delta_1 \Big) - F \Big( y F^{\leftarrow}(1 - t - h) \Big) \Big] \nonumber \\
	& =: {\rm I}^{\sss(1241)} + {\rm I}^{\sss(1242)}.
	\end{align}
The choice  of ${\rm C}_\delta$ in \eqref{eq:vanishing:tail} and $\delta_1$ in \eqref{eq:uniform:continuity:distn:function} imply that ${\rm I}^{\sss(1241)} < \delta$. It is easy to see that ${\rm I}^{\sss(1242)} < 2 \sup  \sup_{0 < h < \delta} \sup_{y \ge {\rm C}_\delta} \overline{F}(y F^{\leftarrow}( 1 - t - h)) = 2 \overline{F}({\rm C}_\delta F^{\leftarrow}( 1- t - \delta)) < \delta$ which follows from the choice of ${\rm C}_\delta$ in \eqref{eq:vanishing:tail}. Therefore, we obtain
	\begin{align}
	{\rm I}^{\sss(124)} < 2 \delta.
	\end{align}
Combining all these facts, we arrive at
	\begin{align}
	{\rm I}_n^{\sss(12)} < 5 \delta \mbox{ almost surely for all } n> N^{\sss(2)}. \label{eq:upperbound:IN:1112}
	\end{align}
\noindent {\bf Step 3: term ${\rm I}_n^{\sss(13)}$}.  We have the almost sure upper bound on ${\rm I}_n^{\sss(13)}$
	\begin{align}
	& \sup_{0 < h < \delta} \sup_{y \ge 1} \Big| F_n \Big( y F^{\leftarrow}(1 - t - h) \Big) - F_n \Big( y F^{\leftarrow}(1 - t) \Big) \Big| \nonumber \\
	& \le \sup_{0 < h < \delta} \sup_{y \ge 1} \Big| F_n \Big( y F^{\leftarrow}(1 - t - h) \Big) - F_n \Big( y F^{\leftarrow}(1 - t) \Big) - F\Big( y F^{\leftarrow}(1 - t) \Big)\nonumber \\
	& \hspace{.5cm} + F \Big( y F^{\leftarrow}(1 - t - h) \Big) \Big|   + \sup_{0 < h < \delta} \sup_{y \ge 1} \Big[ F \Big( y F^{\leftarrow}(1 - t) \Big) - F \Big( y F^{\leftarrow}( 1 - t - h) \Big) \Big] \nonumber \\
	& = {\rm I}_n^{\sss(131)} + {\rm I}^{\sss(132)}.
	\end{align}
We can use the Givenko-Cantelli Theorem once again to show that there exists a large enough integer $N^{\sss(3)}$ such that ${\rm I}_n^{\sss(131)} < \delta$ almost surely for all $n > N^{\sss(3)}$.
We know that $(F^{\leftarrow}(1 - t))_{t \in [\varepsilon, 1]}$ is a uniformly continuous function for every $\varepsilon > 0$. For every $\eta > 0$, there exists a $\lambda$ such that $\sup_{|h| < \lambda} |F^{\leftarrow}(1 - t - h) - F^{\leftarrow}(1 - t)| < \eta.$ We can find a $\widetilde{\delta}$ such that $\sup_{0 <h < \delta} (F^{\leftarrow}( 1 - t) - F^{\leftarrow}( 1 - t - h)) < \widetilde{\delta}$. Therefore, ${\rm I}^{\sss(132)} < \sup_{y \ge 1} \big| F \big( y F^{\leftarrow}( 1 - t) \big) - F\big( y F^{\leftarrow}(1 - t) - y \widetilde{\delta} \big) \big|$. This bound is very similar to the one derived for ${\rm I}^{\sss(124)}$. Hence, the same arguments apply with an appropriately chosen ${\rm C}_{\widetilde{\delta}}$ and we have ${\rm I}^{\sss(132)} < 2\delta$.  Combining these facts, we arrive at
	\begin{align}
	{\rm I}_n^{\sss(13)} < 3 \delta \mbox{ almost surely for all } n \ge N^{\sss(3)}. \label{eq:upperbound:IN:1113}
	\end{align}
Combining \eqref{eq:upperbound:IN:1111}, \eqref{eq:upperbound:IN:1112} and \eqref{eq:upperbound:IN:1113}, we have that $\Iterm{1} \le 13 \delta$ almost surely for all $n \ge {\sf N}_1$, where ${\sf N}_1 = \max(N^{\sss(1)}, N^{\sss(2)}, N^{\sss(3)})$.  Using the fact that $\Iterm{2}$ can be dealt with in a similar way, we see that there exists a sufficiently large integer ${\sf N}_2$ such that, for all $n \ge {\sf N}_2$,
	$$
	\sup_{|h| < \delta} \sup_{y \ge 1} \Big| \overline{F}_n(y X_{n- [nt]:n}) - \overline{F}_n(y X_{n - [n(t + h)] : n}) \Big| < 26 \delta.
	$$
\end{proof}

\begin{proof}[Proof of Proposition~\ref{propn:flt:hills:estimator}]
Using the representation \eqref{eq:hill:uniform:representation} developed in the proof of
Lemma~\ref{lemma:aslimit:hill}, it is enough to establish that
	\begin{align}
	& \sup_{t \in [\vep,1]} \Big| \frac{1}{[nt]} \sum_{i =1}^{[nt]} \log \finv( 1- \unif_{i : n})  - t^{-1} \int_0^t \log \finv(1-s) \dtv s\Big| \asconv 0, \label{eq:flt:hill:decomp:first} \\
	\mbox{ and }\qquad& \sup_{t \in [\vep,1]} \Big| \log \finv( 1 - \unif_{[nt]:n}) - \log \finv(1 - t) \Big| \asconv 0. \label{eq:flt:hill:decomp:second}
	\end{align}
Equation \eqref{eq:flt:hill:decomp:second} can be derived by
combining the Glivenko-Cantelli Theorem for the quantile process
with the continuous mapping theorem (\cite[Theorem~2.3(iii)]{vandervaart:2000}), since $F^{\leftarrow}$ is assumed to be continuous. We are therefore left to prove  \eqref{eq:flt:hill:decomp:first}.
\smallskip


\noindent{\it Proof of \eqref{eq:flt:hill:decomp:first}}.
 We start with the almost sure upper bound of the absolute value in \eqref{eq:flt:hill:decomp:first}:
	\begin{align}
 	& \sup_{t \in [\vep,1]} \Big| \frac{1}{[nt]} \sum_{i=1}^{[nt]} \log F^{\leftarrow} ( 1- \unif_{i : n}) - \frac{1}{[nt]} \sum_{i=1}^n \mbbo(\unif_i \le t) \log F^{\leftarrow} ( 1 - \unif_i)  \Big|  \nonumber \\
	& \hspace{.25cm} + \sup_{t \in [\vep, 1]} \Big| \frac{1}{[nt]} \sum_{i=1}^n \mbbo(\unif_i \le t) \log F^{\leftarrow} ( 1 - \unif_i)  - \frac{1}{t}\int_0^t \log F^{\leftarrow}( 1 - s) \dtv s \Big| \nonumber \\
	& = : {\rm I}^{\sss(1)}_n + {\rm I}^{\sss(2)}_n.
	\end{align}
Below, we use the Glivenko-Cantelli Theorem to show that both terms converge to 0.\\



\noindent {\bf Proof that $\Iterm{1} \asconv 0$}.
It is easy to obtain the almost sure upper bound for $\Iterm{1}$
	\begin{align}
	&  \sup_{t \in [\vep,1]} \frac{1}{[nt]} \sum_{i=1}^n \Big[ \mbbo ( t \le \unif_i \le \unif_{[nt]:n} )
	+  \mbbo(\unif_{[nt]:n} \le \unif_i \le t) \Big] \log F^{\leftarrow}(1 - \unif_i). \label{eq:random:to:deterministic:upper:bound:two}
	\end{align}
Fix $\delta_1 \in (0, \vep)$. Using the Glivenko-Cantelli Theorem for quantile process, there exists  a large integer $N(\delta_1)$ and an event $B$ such that $\prob(B) = 1$ and
	\begin{align*}
	\sup_{t \in [\varepsilon,1]} \Big| \unif_{[nt]: n} - t \Big| \le \delta_1
	\end{align*}
for all $n \ge N(\delta_1)$ on the event $B$. Therefore, on the event $B$, we have the following upper bound for the expression in \eqref{eq:random:to:deterministic:upper:bound:two}:
	\begin{align}
	\sup_{t \in [\vep,1]} \frac{1}{[nt]} \sum_{i=1}^n \mbbo \Big( t - \delta_1 \le \unif_i \le t+ \delta_1 \Big) \log F^{\leftarrow}(1 - t + \delta_1) : = {\rm I}_n^{\sss(11)} \label{eq:random:to:deterministic:upper:bound:three}
	\end{align}
for all $n \ge N(\delta_1)$. In this derivation, we have also used the facts that $F^{\leftarrow}$ is non-decreasing and $\log$ is monotonically increasing. Therefore, to conclude that $\Iterm{1} \asconv 0$, it is enough to show that
	\begin{align}
	{\rm I}_n^{\sss(11)} \asconv 0.
	\end{align}

It is clear that $\sup_{t \in [\vep,1]} \log F^{\leftarrow}(1 - t + \delta_1) = \log F^{\leftarrow}( 1 - \vep + \delta_1) < \infty$ and $\sup_{t \in [\vep,1]} t^{-1} = \vep^{-1}$. Combining these facts, we obtain that, on $B$,
	\begin{align}
	 {\rm I}_n^{\sss(11)} & \le \frac{1}{\vep} \log F^{\leftarrow}(1 - \vep + \delta_1) \sup_{t \in [\vep,1]} \frac{1}{n } \sum_{i=1}^n \mbbo(t - \delta_1 \le \unif_i \le t+ \delta_1) \nonumber \\
	& \hspace{1cm} \asconv \frac{1}{\vep} \Big( \log F^{\leftarrow}( 1 - \vep + \delta_1)  \Big) 2\delta_1, \label{eq:random:to:deterministic:upper:bound:four}
	\end{align}
by the continuous mapping theorem and the Glivenko-Cantelli Theorem. It is clear that the claim ${\rm I}^{\sss(11)}_n \asconv 0$ follows by letting $\delta_1 \to 0$ in \eqref{eq:random:to:deterministic:upper:bound:four}.
\smallskip


\noindent {\it Proof that $\Iterm{2} \asconv 0$}. Note that
	$$
	\frac{1}{t} \int_0^t \log F^{\leftarrow}(1-s) \dtv s = \frac{1}{t} \exptn \Big( \mbbo(\unif \le t ) \log F^{\leftarrow}(1 - \unif) \Big),
	$$
where $\unif$ is a Uniform$[0,1]$ random variable.
We define $f_t (x) = \frac{1}{t} \mbbo(x \le t) F^{\leftarrow}( 1 - x)$ for all $t \in [\vep,1]$. Let ${\cal F}_\vep = \{f_t : t \in [\vep,1]\}$. It has been proved in \eqref{eq:finite:expectation} that $f_t$ is integrable for all $t \in [\vep,1]$.
Note that ${\rm I}^{\sss(2)}_n$ is bounded from above by
	\begin{align*}
	&  \sup_{t \in [\vep,1]} \frac{nt}{[nt]} \sup_{f \in {\cal F}_\vep} \big| \int f \dtv F_n^{\sss (\unif)} - \int f \dtv F^{\sss (\unif)} \big|
	+ \sup_{t \in [\vep,1]} \big| \frac{nt}{[nt]} - 1 \big| \sup_{f \in {\cal F}_\vep} |\int f \dtv F^{\sss (\unif)}| \nonumber \\
	& \le 2 \sup_{f \in {\cal F}_\vep} |\int f \dtv F_n^{\sss (\unif)} - \int f \dtv F^{\sss (\unif)}| + \Big( \int_0^1  \frac{1}{\vep} \log F^{\leftarrow}(1 - x) \dtv x \Big) \sup_{t \in [\vep,1]} \Big| \frac{nt}{[nt]} - 1 \Big| \nn\\
	& := {\rm I}_n^{\sss(21)} + {\rm I}_n^{\sss(22)}
	\end{align*}
almost surely  as $f_t(x) \le \frac{1}{\vep} \log F^{\leftarrow}(1-x) $ for all $t \in [\vep,1]$. It is easy to see that ${\rm I}_n^{\sss(22)}$ is deterministic and ${\rm I}_n^{\sss(22)} \to 0$.

To estimate the term ${\rm I}_n^{\sss(21)}$, we use general empirical process theory. We start by defining the appropriate Glivenko-Cantelli classes. Define, for every measurable and integrable function $f$,
	$$
	\norm{f} = \int_0^1 |f(u)| \dtv u.
	$$
Let ${\cal L}$ denote the class of all functions such that $\norm{f} < \infty$.  Given two functions $l$ and $u$, the bracket $[l,u]$ denotes the class of all functions $f$ such that $l \le f \le u$. A $\delta$-bracket in ${\cal L}$ is a bracket $[l_\delta, u_\delta]$ such that
	\begin{align}
	\int |u_\delta - l_\delta| \dtv F^{\sss (\unif)} \le \delta. \label{eq:delta:bracket:defn}
	\end{align}
Fix $\delta > 0$. The bracketing number ${\cal N}_{[]}(\delta, {\cal F}_\vep, {\cal L})$
is the minimum number of $\delta$-brackets needed to cover ${\cal F}_\vep$. Note that $l_\delta$ and $u_\delta$ may not be elements of ${\cal F}_\vep$. ${\cal F}_\vep$ is said to be an $F^{\sss (\unif)}$-Glivenko-Cantelli class of functions when
	\begin{align}
	\sup_{f \in {\cal F}_\vep} \big| \int f \dtv F_n^{\sss (\unif)} - \int f \dtv F^{\sss (\unif)} \big| \asconv 0. \label{eq:glivenko:cantelli:class:version}
	\end{align}
According to \cite[Theorem~19.4]{vandervaart:2000}, the class of functions ${\cal F}_\vep$ is $F^{\sss (\unif)}$-Glivenko-Cantelli when ${\cal N}_{[]}(\delta, {\cal F}_\vep, {\cal L}) < \infty$ for every $\delta > 0$. Therefore, we need to show that ${\cal N}_{[]}(\delta, {\cal F}_\vep, {\cal L}) < \infty$ for every $\delta>0$ to establish that ${\rm I}_n^{\sss(21)} \asconv 0$.

We complete the proof by providing a finite upper bound ${\rm R}$ on ${\cal N}_{[]}(\delta, {\cal F}_\varepsilon, {\cal L})$. Define ${\cal I}_\vep = \int_0^\vep \log F^{\leftarrow}( 1 - x) \dtv x$ and $M_\vep = \sup_{x \in [\vep,1]} \log F^{\leftarrow}(1 - x)$. Fix ${\rm R}$, and consider a partition $(z_i)_{0 \le i \le {\rm R}}$ such that $\vep = z_0 < z_1 < z_2 < \cdots < z_{\rm R} = 1$ and
	\eqn{
	\label{zi-construction}
	\max_{1 \le i \le {\rm R}}(z_i - z_{i-1}) < \frac{\delta \vep^2}{{\cal I}_\vep + M_\vep}.
	}
Note that ${\rm R}$ may depend on $\delta$, but it is always possible to choose a finite ${\rm R}$ satisfying the above conditions.

By the above partition, $(f_t)_{t \in [\vep,1]} = \bigcup_{i=1}^{\rm R} \{f_t \colon t \in [z_{i-1}, z_i]\}$.  Recall that $f_t (x) = \frac{1}{t} \mbbo(x \le t) F^{\leftarrow}( 1 - x)$ for all $t \in [\vep,1]$. Therefore, $(f_t)_{t \in [z_{i}, z_{i+1}]}\subseteq [l^{\sss(i)}_{\delta}, u_{\delta}^{\sss(i)}],$ where $x\mapsto l_\delta^{\sss(i)}(x)$ and $x\mapsto u_\delta^{\sss(i)}(x)$ that appear in the bracket are defined as
	$$
	l_\delta^{\sss(i)}(x) =\frac{1}{z_{i+1}} \mbbo(x\leq z_{i+1}) \log F^{\leftarrow}(1-x),\quad \mbox{ and } \quad u_\delta^{\sss(i)}(x) =\frac{1}{z_{i}} \mbbo(x\leq z_i) \log F^{\leftarrow}(1-x),
	$$
for all $ 1 \le i \le {\rm R}$, since $F^{\leftarrow}(0) \ge 1$.

To show that ${\cal N}_{[]}(\delta, {\cal F}_\vep, {\cal L}) < {\rm R}$, it is enough to show that $[l_\delta^{\sss(i)}, u_\delta^{\sss(i)}]$ satisfies \eqref{eq:delta:bracket:defn} for every $i \ge 1$. Observe that
	\begin{align}
	&\int [u_\delta^{\sss(i)}(x)-l_\delta^{\sss(i)}(x)]\dtv F^{\sss (\unif)}(x)\nonumber\\
	&= \int_0^1 \Big( \frac{1}{z_{i}} \mbbo(x \le z_{i+1}) \log F^{\leftarrow}(1-x) - \frac{1}{z_{i+1}} \mbbo(x \le z_i) \log F^{\leftarrow}(1-x) \Big) \dtv x \nonumber \\
	& = \frac{z_{i+1} - z_i}{z_i z_{i+1}} \Big( \int_0^\vep \log F^{\leftarrow}( 1- x) \dtv x  + \int_\vep^{z_i} \log F^{\leftarrow}(1- x) \dtv x \Big) \nonumber \\
	& \hspace{2cm} + \frac{1}{z_i} \int_{z_i}^{z_{i+1}} \log F^{\leftarrow}(1 - x) \dtv x \nonumber \\
	& \le \frac{z_{i+1} - z_i}{\vep^2} ( {\cal I}_\vep + ( 1 - \vep) M_\vep) + \frac{z_{i +1} - z_i}{\vep} M_\vep  < \delta,
	\end{align}
where the final inequality follows by \eqref{zi-construction}.
This implies that ${\cal N}_{[]}(\delta, {\cal F}_\vep, {\cal L}) < {\rm R} < \infty$, and we conclude that ${\cal F}_\vep$ is a $F^{\sss (\unif)}$-Glivenko-Cantelli class. Hence, \eqref{eq:glivenko:cantelli:class:version} holds.
This in turn implies that ${\rm I}_n^{\sss(21)}\asconv 0.$
\end{proof}


\subsection{Proofs of Lemmas~\ref{lem-inf-t[vep,1]} and \ref{lemma:sublinear:growth:noteventually:pareto}}
\label{sec-proofs-lemma47-48}
\begin{proof}[Proof of Lemma~\ref{lem-inf-t[vep,1]}]
Proposition~\ref{prop:tightness:suprema} implies tightness of the sequence $(\overline{Z}_n(t))_{t \in [\vep, 1]}$ by \cite[Theorem~8.2 ]{billingsley:2013}. We have derived the convergence of finite-dimensional distributions in Proposition~\ref{propn:dnk:functional:limit}. According to \cite[Theorem~13.4.1]{whitt2002stochastic}, the supremum is a continuous functional in the $J_1$-topology. This completes the proof of  Lemma~\ref{lem-inf-t[vep,1]}.
\end{proof}
			
\begin{proof}[Proof of Lemma~\ref{lemma:sublinear:growth:noteventually:pareto}]
We first note that
	\begin{align}
	Z \big(\overline{F}(x), y \big) & = y^{- \alpha} \frac{L(yx)}{L(x)} - y^{ - \overline{F}(x) \Big( \int_0^{\overline{F}(x)} \log \frac{F^{\leftarrow}(1 - s)}{x} \dtv s \Big)^{-1}} \nonumber \\
	& =: y^{- \alpha} \Big[ \frac{L(yx)}{L(x)} - y^{ {\cal U}(x)} \Big] := y^{- \alpha} K(x,y),
	\end{align}
where ${\cal U}(x)$ is defined in \eqref{eq:defn:calu}.
%
Define $\mathbb{V} := \{x \in (0, \infty) \colon {\cal U}(x) = 0\}$. There are two possibilities, depending on whether $\mathbb{V} = \varnothing$ or $\mathbb{V} \neq \varnothing$. We deal with each of these two cases separately. The former case will be divided into two subcases and we directly verify \eqref{eq:aim:positive:limit} for each of the subcases, whereas a contradiction argument will be used to deal with the latter case.\\

\noindent {\bf Case-I: ${\mathbb V} = \varnothing$}. Our aim is to show that, for every $\varepsilon > 0$,
	\begin{align}
 	\inf_{x \in [F^{\leftarrow}(0), F^{\leftarrow}(1 - \varepsilon)]} \sup_{y \ge 1}| Z(\overline{F}(x), y)| > 0. \label{eq:aim:positive:limit}
	\end{align}
There are two possibilities, depending on whether ${\cal U}(x)>0$ or ${\cal U}(x) < 0$ for all $x \in [F^{\leftarrow}(0), F^{\leftarrow}(1 - \varepsilon)]$. Indeed, suppose that ${\cal U}(x)$ changes its sign when $x \in [F^{\leftarrow}(0), F^{\leftarrow}(1- \varepsilon)]$. Then, by continuity of $x\mapsto {\cal U}(x)$, we conclude that $\mathbb{V} \neq \varnothing$, which is a contradiction to our assumption. If we assume that ${\cal U}(x)$ is positive for all $x \in [F^{\leftarrow}(0), F^{\leftarrow}(1 - \varepsilon)]$, then $\mathbb{V} = \varnothing$ implies that
	\begin{align}
	\inf_{x \in [F^{\leftarrow}(0), F^{\leftarrow}(1- \varepsilon)]}  {\cal U}(x) >0. \label{eq:positive:ux}
	\end{align}

Suppose instead that ${\cal U}(x)$ is negative for all $x \in [F^{\leftarrow}(0), F^{\leftarrow}(1 - \varepsilon)]$. Then $\mathbb{V} = \varnothing$ implies that
	\begin{align}
	\sup_{x \in [F^{\leftarrow}(0), F^{\leftarrow}(1 - \varepsilon)]} {\cal U}(x) < 0. \label{eq:negative:ux}
	\end{align}
We shall deal with each of these two conditions separately.
\smallskip

\noindent{\bf Assume that \eqref{eq:positive:ux} holds}.  Define $\varrho := \inf_{x \in [F^{\leftarrow}(0), F^{\leftarrow}(1- \varepsilon)]} {\cal U}(x) >0$. As $F^{\leftarrow}(0) \ge 1$, we can use Potter's bound (see \cite[Proposition~0.8(ii)]{resnick:1986}) for $L(yx)$ for large enough $y$. Fix $\varsigma \in (0, \varrho)$. Then there exists a large positive number $y_0$ such that
	\begin{align}
	\frac{L(yx)}{L(x)} \le \frac{(yx)^{\varsigma}}{\inf_{u \in [F^{\leftarrow}(0), F^{\leftarrow}(1- \varepsilon)]} L(u)}
	\le y^{\varsigma} \frac{\big( F^{\leftarrow}(1 - \varepsilon) \big)^{\varsigma}}{\inf_{u \in [F^{\leftarrow}(0), F^{\leftarrow}(1- \varepsilon)]} L(u)}
	\label{eq:lower:bound:ratio:slowlyvarying}
	\end{align}
for all $x \in [F^{\leftarrow}(0), F^{\leftarrow}(1- \varepsilon)]$ and $y \ge y_0$. It is easy to see that, for every $y \ge 1$,
	\begin{align}
	\exp \Big\{  {\cal U}(x) \log y \Big\} \ge y^{\varrho}.
	\label{eq:lower:bound:exp}
	\end{align}
Combining \eqref{eq:lower:bound:ratio:slowlyvarying} and \eqref{eq:lower:bound:exp}, we obtain that
	\begin{align}
	&\inf_{x \in [F^{\leftarrow}(0), F^{\leftarrow}(1 - \varepsilon)]} y^{- \alpha}|K(x,y)| \nonumber\\
	& = \inf_{x \in [F^{\leftarrow}(0), F^{\leftarrow}(1- \varepsilon)]} y^{- \alpha} \max \big( K(x,y), - K(x,y) \big) \nonumber \\
	& \ge \inf_{x \in [F^{\leftarrow}(0), F^{\leftarrow}(1- \varepsilon)]} y^{- \alpha} \big(-K(x,y) \big) \nonumber \\
	& \ge y^{\varsigma - \alpha} \Big( y^{\varrho - \varsigma} - \frac{\big( F^{\leftarrow}(1- \varepsilon) \big)^{\varsigma}}{\inf_{u \in [F^\leftarrow(0), F^{\leftarrow}(1 - \varepsilon)]} L(u)} \Big). 	
	\label{eq:lower:bound:kxy:positive}
	\end{align}
It is easy to see that we can choose $y$ large enough so that the lower bound derived in \eqref{eq:lower:bound:kxy:positive} becomes positive and hence \eqref{eq:aim:positive:limit} holds.
\smallskip

\noindent{\bf Assume that \eqref{eq:negative:ux} holds}. Define $\varrho := \sup_{x \in [F^{\leftarrow}(0), F^{\leftarrow}(1 - \varepsilon)]}  {\cal U}(x) < 0$. We can use Potter's bound once again.  Fix $\varsigma \in (0, -\varrho)$. This means that there exists an $y_0 \ge 1$ such that, for all $x \in [F^{\leftarrow}(0), F^{\leftarrow}(1 - \varepsilon)]$ and $y \ge y_0$ as $F^{\leftarrow}(0) \ge 1$,
	\begin{align}
	\frac{L(xy)}{L(x)} \ge y^{- \varsigma} \frac{\big( F^{\leftarrow}(1 - \varepsilon) \big)^{- \varsigma}}{\sup_{u \in [F^{\leftarrow}(0), F^{\leftarrow}(1 - \varepsilon)]} L(u)}.
	\label{eq:ratio:slowlyvarying:lowerbound:negative}
	\end{align}
It is clear that, for every $y \ge 1$ and $x \in [F^{\leftarrow}(0), F^{\leftarrow}(1 - \varepsilon)]$,
	\begin{align}
	\exp \Big\{   {\cal U}(x) \log y \Big\} \le y^{ \varrho},
	\label{eq:exponential:lowerbound:negative}
	\end{align}
where we recall that $\varrho<0$. Combining \eqref{eq:ratio:slowlyvarying:lowerbound:negative} and \eqref{eq:exponential:lowerbound:negative}, we obtain that
	\begin{align}
	&\inf_{x \in [F^{\leftarrow}(0), F^{\leftarrow}(1 - \varepsilon)]} y^{-\alpha} |K(x,y)| \nonumber\\
	& = \inf_{x \in [F^{\leftarrow}(0), F^{\leftarrow}(1 - \varepsilon)]} y^{- \alpha}\max( K(x,y), - K(x,y)) \nonumber \\
	& \ge \inf_{x \in [F^{\leftarrow}(0), F^{\leftarrow}(1 - \varepsilon)]}  y^{- \alpha} K(x,y) \nonumber\\
	& \ge y^{\varsigma - \alpha} \Big( y^{- \varrho - \varsigma} \frac{\big( F^{\leftarrow}(1 - \varepsilon) \big)^{- \varsigma}}{ \sup_{u \in [F^{\leftarrow}(0), F^{\leftarrow}(1 - \varepsilon)]} L(u)}  - 1 \Big).  	
	\label{eq:negative:lowerbound:kxy}
	\end{align}
Note that since $(- \varrho - \varsigma) >0$, we can make the right hand side of \eqref{eq:negative:lowerbound:kxy} positive by choosing $y$ large enough, and hence \eqref{eq:aim:positive:limit} holds. \\

\noindent{\bf Case-II: ${\mathbb V} \neq \varnothing$}. 
Note that $x \mapsto  (|Z(\overline{F}(x),y)|)_{y\geq 1}$ is a collection of continuous functions. Invoking \cite[Proposition 1.26(a)]{rockafellar:wets:1998}, we conclude that $x \mapsto \overline{Z}(\overline{F}(x)) :=\sup_{y \ge 1} |Z(\overline{F}(x),y)|$ is a lower semicontinuous function.
%
%
%

We now suppose that $\inf_{x \in [F^{\leftarrow}(0), F^{\leftarrow}(1 - \vep)]} \sup_{y \ge 1}| Z(\overline{F}(x), y)| = 0$ for some $\vep > 0$. Using lower semicontinuity of $x\mapsto \overline{Z}(\overline{F}(x))$ and the fact that  $[F^{\leftarrow}(0), F^{\leftarrow}(1 - \vep)]$ is a compact set, the Bolzano-Weierstrass Theorem implies that there exists an $x_0$ such that $\sup_{y \ge 1}| Z(\overline{F}(x_0), y)| = 0$. Then we must have that $x_0 \in \mathbb{V}$ i.e., ${\cal U}(x_0) = 0$. In that case,
	\begin{align}
	\sup_{y \ge 1} y^{- \alpha}\Big| \frac{L(x_0 y)}{L(x_0)} - 1  \Big| = 0,
	\end{align}
which implies that $L(y) = L(x_0)$ for all $y \ge x_0$ and hence $F$ is eventually Pareto. This is a contradiction to the assumption in Theorem~\ref{thm-sublinear} and so \eqref{eq:aim:positive:limit} follows in this case.
\end{proof}

\section{Eventually Pareto case: Proof of Proposition \ref{PROP-EVENT-PARETO}}
\label{app-eventually-Pareto}
In this appendix, we prove Proposition \ref{PROP-EVENT-PARETO}. We argue by contradiction as follows.
Recall $x_0$ in Definition \ref{def-eventually-Pareto}. If we can show that the limit of $D_{n, \kappa_n^\star}$ is positive on the event $\{\kappa_n^\star \ge n \overline{F}(x_0 - \vep)\}$, then we obtain a contradiction to $D_{n,\kappa_n^\star} \probconv 0$ (recall Theorem~\ref{thm:probconv:DKN}). Therefore, we conclude $\lim_{n \to \infty} \prob(\kappa_n^\star \ge n \overline{F}(x_0 - \vep)) = 0$ following the same argument as in \eqref{eq:upper:bound:reduced:aim}.  The rest of the proof is dedicated to the proof that the limit of $D_{n, \kappa_n^\star}$ is positive on the event $\{\kappa_n^\star \ge n \overline{F}(x_0 - \vep)\}$ when $F$ is eventually Pareto after $x_0$, but not before $x_0$.

By Lemma \ref{lem-inf-t[vep,1]} (recall also \eqref{Dn-conv} in Theorem \ref{thm-sublinear}),
	\eqn{
	\inf_{x\leq x_0-\vep} D_{n_k,[n_k\bar{F}(x)]}\probconv \inf_{x\leq x_0-\vep} \sup_{y\geq 1} y^{-\alpha}\Big|y^{{\cal U}(x)}-\frac{L(xy)}{L(x)}\Big|. \label{eq:lim:dist:eventially:pareto}
	}
Our aim is to show that the limit is strictly positive by \eqref{positive-Dn} (recall also the proof of Lemma \ref{lemma:sublinear:growth:noteventually:pareto} in Section \ref{sec-proofs-lemma47-48}), since $\bar{F}(x)$ is {\em not} Pareto for $x\leq x_0-\vep$.  Note that $x_0$ is the smallest positive number satisfying
	\begin{align}
	\sup_{x \ge x_0} |L(x) - L(x_0)| = 0. \label{eq:minimality:defn:evetually:pareto}
	\end{align}
We shall obtain a contradiction to the minimality of $x_0$ if the right hand side of \eqref{eq:lim:dist:eventially:pareto} vanishes.

We assume that the right hand side of \eqref{eq:lim:dist:eventially:pareto} equals $0$. It has been proved in {\bf Case-I} of the proof of Lemma~\ref{lemma:sublinear:growth:noteventually:pareto} that the limit is positive if ${\mathbb V} = \{ x \ge [1, x_0 - \vep]\colon {\cal U}(x) = 0\} = \varnothing$ (see \eqref{eq:defn:calu} for ${\cal U}(x)$).
Therefore, when $\inf_{x\leq x_0-\vep} D_{n_k,[n_k\bar{F}(x)]}\probconv 0$, we must have that ${\mathbb V} \neq \varnothing$. We now use the arguments in { \bf Case-II} of the proof of Lemma~\ref{lemma:sublinear:growth:noteventually:pareto} to prove the existence of an element $x_1 \in {\mathbb V}$ such that
	\begin{align}
 	\sup_{y \ge 1} y^{-\alpha} \Big| y^{{\cal U}(x_1)} - \frac{L(yx_1)}{L(x_1)} \Big| = 0. \label{eq:equation:x1}
	\end{align}
As $x_1 \in {\mathbb V}$, we have that ${\cal U}(x_1) = 0$,  so that \eqref{eq:equation:x1} reduces to
	\begin{align}
	\sup_{x \ge x_1} |L(x) - L(x_1)| = 0. \label{eq:contradiction:minimality}
	\end{align}
Thus, \eqref{eq:contradiction:minimality} contradicts the minimality of $x_0$ in \eqref{eq:minimality:defn:evetually:pareto} as $x_1 \in [1, x_0 - \vep]$ and $\vep > 0$. Hence, the limit in \eqref{eq:lim:dist:eventially:pareto} must be positive. This proves Proposition \ref{PROP-EVENT-PARETO}.
\qed


\end{document}